\documentclass[11pt,reqno]{amsart}

\usepackage{amssymb, amsmath, amsthm}
\usepackage{hyperref}
\usepackage[alphabetic,lite]{amsrefs}
\usepackage{verbatim}
\usepackage{amscd}   
\usepackage[all]{xy} 
\usepackage{youngtab} 
\usepackage{young} 
\usepackage{ytableau}
\usepackage{tikz}
\usepackage{ mathrsfs }
\usepackage{cases}
\usepackage{array}

\newcommand{\defi}[1]{{\upshape\sffamily #1}}

\renewcommand{\a}{\alpha}
\renewcommand{\b}{\beta}

\newcommand{\D}{\mathcal{D}}
\newcommand{\E}{\mathcal{E}}

\newcommand{\G}{\operatorname{G}}
\newcommand{\kk}{{\mathbf k}}

\newcommand{\LL}{\Lambda}

\newcommand{\onto}{\twoheadrightarrow}
\newcommand{\oo}{\otimes}

\renewcommand{\P}{\mathcal{P}}
\newcommand{\s}{\sigma}
\renewcommand{\S}{\mathfrak{S}}
\renewcommand{\t}{\underline{t}}
\newcommand{\w}{\underline{w}}

\let\uuu\u 
\renewcommand{\u}{\underline{u}}

\renewcommand{\v}{\underline{v}}
\newcommand{\x}{\underline{x}}
\newcommand{\y}{\underline{y}}
\newcommand{\z}{\underline{z}}

\newcommand{\X}{\mathcal{X}}
\newcommand{\Y}{\mathcal{Y}}

\newcommand{\Z}{\mathcal{Z}}

\newcommand{\Ann}{\operatorname{Ann}}

\newcommand{\Ext}{\operatorname{Ext}}
\newcommand{\GL}{\operatorname{GL}}

\newcommand{\Sym}{\operatorname{Sym}}
\newcommand{\Tor}{\operatorname{Tor}}

\newcommand{\codim}{\operatorname{codim}}
\newcommand{\coker}{\operatorname{coker}}

\renewcommand{\ker}{\operatorname{ker}}

\newcommand{\pdim}{\operatorname{pdim}}
\newcommand{\reg}{\operatorname{reg}}

\newcommand{\bb}[1]{\mathbb{#1}}

\renewcommand{\rm}[1]{\textrm{#1}}
\newcommand{\mc}[1]{\mathcal{#1}}
\newcommand{\mf}[1]{\mathfrak{#1}}

\newcommand{\op}[1]{\operatorname{#1}}

\newcommand{\ul}[1]{\underline{#1}}

\def\lra{\longrightarrow}

\newtheorem{theorem}{Theorem}[section]
\newtheorem*{theorem*}{Theorem}
\newtheorem*{problem*}{Problem}
\newtheorem{lemma}[theorem]{Lemma}

\newtheorem{proposition}[theorem]{Proposition}
\newtheorem{corollary}[theorem]{Corollary}
\newtheorem*{corollary*}{Corollary}

\newtheorem*{main-thm*}{Main Theorem}
\newtheorem*{linear-resolutions*}{Theorem on Linear Resolutions}
\newtheorem*{reg-pdim*}{Theorem on Regularity and Projective Dimension}
\newtheorem*{injectivity-Ext*}{Theorem on Injectivity of Maps from Ext to Local Cohomology}
\newtheorem*{reg-pows*}{Theorem on Regularity of Powers}
\newtheorem*{invariant-chains*}{Theorem on Invariant Chains of Ideals}

\theoremstyle{definition}
\newtheorem{definition}[theorem]{Definition}
\newtheorem*{definition*}{Definition}
\newtheorem{example}[theorem]{Example}
\newtheorem{problem}[theorem]{Problem}

\theoremstyle{remark}
\newtheorem{remark}[theorem]{Remark}
\newtheorem*{remark*}{Remark}

\numberwithin{equation}{section}

\newcommand{\claudiu}[1]{{\color{red} \sf $\clubsuit\clubsuit\clubsuit$ Claudiu: [#1]}}

\begin{document}

\title{Regularity of $\S_n$-invariant monomial ideals}

\author{Claudiu Raicu}
\address{Department of Mathematics, University of Notre Dame, 255 Hurley, Notre Dame, IN 46556\newline
\indent Institute of Mathematics ``Simion Stoilow'' of the Romanian Academy}
\email{craicu@nd.edu}

\subjclass[2010]{Primary 13D07, 05E40, 13D45}

\date{\today}

\keywords{Monomial ideals, regularity, projective dimension, $\Ext$ modules, local cohomology}

\begin{abstract} For a polynomial ring $S$ in $n$ variables, we consider the natural action of the symmetric group $\S_n$ on $S$ by permuting the variables. For an $\S_n$-invariant monomial ideal $I\subseteq S$ and $j\geq 0$, we give an explicit recipe for computing the modules $\Ext^j_S(S/I,S)$, and use this to describe the projective dimension and regularity of $I$. We classify the $\S_n$-invariant monomial ideals $I$ that have a linear free resolution, and also characterize those which are Cohen--Macaulay. We then consider two settings for analyzing the asymptotic behavior of regularity: one where we look at powers of a fixed ideal $I$, and another where we vary the dimension of the ambient polynomial ring and examine the invariant monomial ideals induced by $I$. In the first case we determine the asymptotic regularity for those ideals $I$ that are generated by the $\S_n$-orbit of a single monomial by solving an integer linear optimization problem. In the second case we describe the behavior of regularity for any $I$, recovering a recent result of Murai.
\end{abstract}

\maketitle

\section{Introduction}\label{sec:intro}

We let $S = \kk[e_1,\cdots,e_n]$ be a polynomial ring in $n$ variables over a field $\kk$. We let $\S_n$ denote the symmetric group on $n$ letters, acting on $S$ by permutations of the variables. If $\x=(x_1,\cdots,x_n)\in\bb{Z}^n_{\geq 0}$, we write
\[ e^{\x} = e_1^{x_1}\cdots e_n^{x_n}\]
for the corresponding monomial in $S$. We let $I_{\x}$ denote the ideal in $S$ generated by the $\S_n$-orbit of $e^{\x}$:
\begin{equation}\label{eq:def-Ix}
 I_{\x} = \langle \s(e^{\x}) : \s\in\S_n\rangle.
\end{equation}
We say that $\x$ is a \defi{partition} (or that it is \defi{dominant}) if $x_1\geq\cdots\geq x_n$, and let $\P_n$ denote the set of all partitions in $\bb{Z}^n_{\geq 0}$. Observe that $I_{\x} = I_{\y}$ for a unique $\y\in\P_n$, obtained by arranging the entries of $\x$ in non-increasing order. Every $\S_n$-invariant ideal $I\subseteq S$ is determined by a (finite) subset $\X\subset\P_n$ via $I=I_{\X}$, where
\begin{equation}\label{eq:def-IX}
I_{\mc{X}} = \sum_{\x\in\mc{X}} I_{\x}.
\end{equation}
The goal of this paper is to study homological invariants of the ideals $I_{\X}$, for which the following definition will play a fundamental role. Recall that if $\x\in\P_n$ then $\x'$ denotes the \defi{conjugate partition}, where $x'_i$ counts the number of parts of $\x$ with $x_j\geq i$. If $\x,\y\in\P_n$, we write $\x\geq\y$ if $x_i\geq y_i$ for all $i$. If $\x\in\P_n$ and $c\geq 0$ we write $\x(c)$ for the partition whose $i$-th part is $x_i(c)=\min(x_i,c)$.

\begin{definition}\label{def:ZX}
 For a subset $\mc{X}\subset\P_n$ we define $\mc{Z}(\mc{X})$ to be the set consisting of pairs $(\z,l)$, where $\z\in\P_n$ and $l\geq 0$ are such that if we write $c=z_1$ then the following hold:
 \begin{enumerate}
  \item There exists a partition $\ul{x}\in\mc{X}$ such that $\ul{x}(c)\leq\z$ and $x'_{c+1}\leq l+1$.
  \item For every partition $\ul{x}\in\mc{X}$ satisfying (1) we have $x'_{c+1}=l+1$.
 \end{enumerate}
\end{definition}

\begin{reg-pdim*}
 For every subset $\X\subset\P_n$ with $I_{\X}\neq S$ we have
 \begin{equation}\label{eq:reg-pdim-IX}
  \reg(I_{\X}) = \max \left\{ |\z| + l + 1 : (\z,l) \in \Z(\X)\right\}\mbox{ and }\pdim(I_{\X}) = \max \left\{ n-1-l : (\z,l) \in \Z(\X)\right\},
 \end{equation}
 where $\reg(-)$ (resp. $\pdim(-)$) denotes \defi{Castelnuovo--Mumford regularity} (resp. \defi{projective dimension}).
\end{reg-pdim*}

To simplify notation, we will omit trailing zeros from a partition $\x$, and write for instance $(5,1)$ instead of $(5,1,0,0)$ when $n=4$. For visualization purposes, often in examples we will draw a \defi{Young diagram} instead of writing the entries of the corresponding partition: for instance $\x=(4,2,1)$ will be pictured as
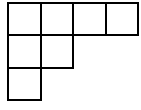
\[\yng(4,2,1)\]
We will also write $\emptyset$ for the \defi{empty partition}, all of whose parts are equal to $0$. Notice that the partition $\x(c)$ considered in Definition~\ref{def:ZX} is the one formed by the first $c$ columns of the Young diagram of $\x$. With these conventions, we now illustrate the Theorem on Regularity and Projective Dimension with an example.

\begin{example}\label{ex:reg-pdim}
Consider the case when $n=3$ and $\ytableausetup{smalltableaux,aligntableaux=center} \X = \{(2,1,1),(4,2)\}=\left\{\ydiagram{2,1,1},\ydiagram{4,2}\right\}$. We have that
\[ I_{\X} = \langle e_1^2 e_2 e_3,e_1 e_2^2 e_3,e_1 e_2 e_3^2, e_1^4 e_2^2,e_1^4 e_3^2,e_2^4 e_1^2,e_2^4 e_3^2,e_3^4 e_1^2,e_3^4 e_2^2\rangle,\]
and a Macaulay2 \cite{M2} calculation finds that the Betti table of $I_{\X}$ is (recall that the Betti number $\b_{i,i+j} = \dim_{\kk}\Tor^S_i(I_{\X},S)_{i+j}$ is placed in row $j$, column $i$, and that a dash indicates a vanishing Betti number)
\[
 \begin{array}{c|ccc}
      &0&1&2\\ \hline
      \text{4}&3&3&1\\ 
      \text{5}&-&-&-\\
      \text{6}&6&6&-\\
      \text{7}&-&3&3\\
\end{array}
\]
In particular, this shows that $\reg(I_\X)=7$ and $\pdim(I_\X) = 2$. Using Definition~\ref{def:ZX} we find the following table, whose first row describes the elements of $\Z(\X)$:
\[
\setlength{\extrarowheight}{2pt}
\ytableausetup{smalltableaux,aligntableaux=center}
\begin{array}{c|c|c|c|c|c|c}
(\z,l) & \left(\emptyset,1\right) & \left(\ydiagram{1,1},1\right) & \left(\ydiagram{1,1,1},0\right) & \left(\ydiagram{2,2},0\right) & \left(\ydiagram{3,2},0\right) & \left(\ydiagram{3,3},0\right) \\[1em]
 \hline 
|\z|+l+1 & 2 & 4 & 4 & 5 & 6 & 7 \\ \hline
n-1-l & 1 & 1 & 2 & 2 & 2 & 2 \\
\end{array}
\]
It follows that as $(\z,l)$ varies in $\Z(\X)$, the maximum value of $|\z|+l+1$ is $\reg(I_\X)$, and the maximum value of $n-1-l$ is $\pdim(I_\X)$, as predicted by (\ref{eq:reg-pdim-IX}).
\end{example}

To put the Theorem on Regularity and Projective Dimension in context, we analyze a few special cases. We make one more convention: for partitions with repeating parts, we use the abbreviation $(b^a)$ for the sequence $(b,b,\cdots,b)$ of length $a$; for instance $(3,3,3,3,1,1)$ will be written as $(3^4,1^2)$. 

\smallskip

\noindent {\bf The ideals $I_{\x}$.} When $\X=\{\x\}$ is a singleton, one has using Definition~\ref{def:ZX} that
\begin{equation}\label{eq:Zx}
 \Z(\X) = \{(\z,l)\in\P_n\times\bb{Z}_{\geq 0} : \mbox{there exists }0\leq c\leq x_1-1\mbox{ such that }z_1=c,\ \z\geq\x(c),\ l=x'_{c+1}-1\}.
\end{equation}
It follows that for $(\z,l)\in\Z(\X)$, the quantities $|\z|+l+1$ and $n-l-1$ are both maximized when $c=x_1-1$ and $\z=(c^n)$. This shows that
\begin{equation}\label{eq:reg-pdim-Ix}
 \reg(I_{\x}) = n\cdot(x_1-1) + x'_{x_1},\mbox{ and }\pdim(I_{\x}) = n - x'_{x_1}.
\end{equation}

In the case when $\x$ has distinct parts ($x_1>x_2>\cdots>x_n$), the ideals $I_{\x}$ are known as \defi{permutohedron ideals}, and their minimal free resolution is constructed explicitly as a cellular resolution (see \cite[Section~4.3.3]{mil-stu} or \cite{bayer-sturmfels}). When $\x$ has repeated parts, the cellular resolution is no longer minimal, but the Betti numbers of $I_{\x}$ can still be determined \cite{kumar-kumar}. One can then also derive (\ref{eq:reg-pdim-Ix}) from the explicit knowledge of the (non-)vanishing behavior of the Betti numbers of $I_{\x}$.

\smallskip

\noindent {\bf Square-free ideals.} The square-free $\S_n$-invariant monomial ideals have a simple classification -- we have one for each $p=1,\cdots,n$, which is denoted $I_p$ and is generated by all the square-free monomials of degree~$p$. Using the earlier notation, we have $I_p = I_{(1^p)}$ for each $p$. Using (\ref{eq:reg-pdim-Ix}) we obtain $\reg(I_p)=p$ and $\pdim(I_p)=n-p$, that is, $I_p$ is Cohen--Macaulay with a linear resolution. In addition to the well-understood Betti numbers of~$I_p$, we note that the action of $\S_n$ on the minimal resolution of $I_p$ was described in \cite{galetto}. It would be interesting to understand more generally the minimal resolutions of the ideals $I_{\X}$ (see \cite{murai} for their asymptotic behavior).

\smallskip

\noindent {\bf Polymatroidal ideals.} The products $I = I_{p_1} \cdot I_{p_2} \cdots$ are examples of \defi{polymatroidal ideals}, as defined in \cite{conca-herzog}, and in particular they have a linear resolution. If we assume that $p_1\geq p_2\geq\cdots$, we can form a partition $\x$ by declaring that $x'_i = p_i$ for all $i$. We can then write $I=I_{\X}$ where $\X$ is the set of all partitions $\y\in\P_n$ that have size $|\x|$ and are \defi{dominated} by $\x$ (that is, $y_1+\cdots+y_i\leq x_1+\cdots+x_i$ for all $i$). We invite the reader to check, using Definition~\ref{def:ZX}, that every $(\z,l)\in\Z(\X)$ satisfies $|\z|+l+1\leq |\x|$, and equality is attained if we take $c=x_1-1$, $\z=\x(c)$ and $l=x'_{c+1}-1$. Using (\ref{eq:reg-pdim-IX}) this shows that $\reg(I) = |\x|$, providing an alternative verification that $I$ has a linear resolution (see also the discussion on symmetric shifted ideals).

\smallskip

Specializing the discussion above to the case $p_1=p_2=\cdots$, we see that $\reg(I_p^d)=p\cdot d$ for all $d\geq 1$. This is a very special instance of a general phenomenon, discovered in \cites{cutkosky-herzog-trung,kodiyalam}, which asserts that for an arbitrary homogeneous ideal $I$, the regularity of $I^d$ is computed by a linear function $a\cdot d + b$ for $d\gg 0$. If $I$ is generated in a single degree $r$ then one has $a=r$, but the constant term $b$ is in general quite mysterious. For ideals of minors of a generic matrix, this constant was studied in \cite{raicu-reg-coh}. The corresponding problem for ideals of Pfaffians is resolved in \cite{perlman}. There is an extensive literature analyzing the case when $I$ is a monomial edge ideal (see \cites{nevo-peeva,ban,BHT} and the references therein). The theorem below computes the constant term $b$ in the case when $I=I_{\w}$ for every $\w\in\P_n$.

\begin{reg-pows*}
If we write the conjugate partition to $\w$ as $\w' = (n^{a_0},h_1^{a_1},h_2^{a_2},\cdots,h_k^{a_k})$ with $n>h_1>\cdots>h_k>0$, then we have that $\reg(I_{\w}^d) = d\cdot |\w| + b$ for $d\gg 0$, where
\[b = (n-h_1)\cdot(a_1-1) + (h_1-h_2)\cdot (a_2-1) + \cdots + (h_{k-1}-h_k)\cdot (a_k-1) + (h_k-1)\cdot (a_k-1).\]
 In particular, the powers $I_{\w}^d$ have a linear resolution for $d\gg 0$ if and only if $a_1=\cdots=a_k=1$, that is, if and only if $w_i-w_{i+1}\leq 1$ for all $i=1,\cdots,n-1$.
\end{reg-pows*}

In light of (\ref{eq:reg-pdim-IX}), finding the exact value of the regularity of an $\S_n$-invariant monomial ideal $I$ amounts to solving a linear integer optimization problem. For $I=I_{\w}^d$, this problem is a high-multiplicity partitioning problem, which is an instance of a resource-allocation problem that is fundamental in Operations Research. In \cite{raicu-partitioning} we have found essentially optimal criteria for the feasibility of this optimization problem when $d\gg 0$, and we apply the results established there to derive a proof of the Theorem on Regularity of Powers.

\smallskip

\noindent {\bf Symmetric (strongly) shifted ideals.} In \cite{symmetric-shifted}, the authors study a class of $\S_n$-invariant monomial ideals, called \defi{symmetric shifted}, along with the subclass of \defi{symmetric strongly shifted} ideals (see Section~\ref{sec:linear} for the terminology). They show that these ideals have a linear free resolution, describe their Betti numbers, and leave open the question of classifying the $\S_n$-invariant monomial ideals that have a linear resolution. We answer their question below, and also identify an interesting class of symmetric strongly shifted ideals.

\begin{linear-resolutions*}
 An $\S_n$-invariant monomial ideal $I$ has a linear free resolution if and only if it is symmetric shifted. If $I_{\w}^d$ has a linear resolution for $d\gg 0$ then it is symmetric strongly shifted.
 \end{linear-resolutions*}

\smallskip

\noindent {\bf $\S_n$-invariant ideals for varying $n$.} A problem that has attracted much interest in recent years (in the context of representation stability, FI-modules, Noetherianity up to symmetry) is concerned with the study of chains of (not necessarily monomial) ideals $(I_n)_{n\geq 1}$ with $I_n \subset \kk[e_1,\cdots,e_n]$ being $\S_n$-invariant, and
\begin{equation}\label{eq:invariant-chain}
 \S_m(I_n) \subseteq I_m \mbox{ for }m\geq n.
\end{equation}
It is conjectured (in a slightly more general setting) in \cites{LNNR1,LNNR2} that $\reg(I_n)$ and $\pdim(I_n)$ are eventually described by linear functions on $n$. 

It is known that for chains of $\S_n$-invariant ideals as above, the inclusions (\ref{eq:invariant-chain}) are equalities for $m>n\gg 0$ \cite{cohen,AH,HS}, so they depend on a finite amount of information. When the ideals $I_n$ are monomial, this information is simply a finite set of partitions, as follows. Since $n$ varies, it is convenient now to regard $\mc{P}_n$ as a subset of $\mc{P}_{n+1}$ by appending a zero to any $n$-tuple $\x\in\P_n$ to get a tuple in $\P_{n+1}$. We write $\P=\bigcup_{n}\P_n$ for the set of all partitions, and given any subset $\X\subset\P$, we write
\[ \X_n = \{\x\in\X : \x\mbox{ has at most }n\mbox{ parts}\},\]
and we view $\X_n$ as a subset of $\P_n$ in the natural way. Every chain $(I_n)_{n\geq 1}$ of $\S_n$-invariant monomial ideals has the property that there exists a finite subset $\X\subset\P$ with $I_n=I_{\X_n}$ for $n\gg 0$. We have the following.

\begin{invariant-chains*}
 Let $\X$ denote a finite non-empty set of pairwise incomparable partitions, and define
 \begin{equation}\label{def:mwW}
  m = \max\{ i : x_i\neq 0\mbox{ for some }\x\in\X\},\quad w = \min\{x_1 : \ul{x}\in\mc{X}\},\quad\mbox{ and }\quad W = \max\{x_1 : \ul{x}\in\mc{X}\}.
 \end{equation}
 If we let $\Y = \{ \x-\x(w-1) : \x\in\X\}$, then we have the following.
\begin{enumerate}
 \item There exists a constant $C$ such that $\reg(I_{\Y_n})=C$ for $n\geq m$.
 \item We have $\reg(I_{\X_n}) = (w-1)\cdot n + C$ for $n\geq \max\bigl(m,(m-1)\cdot(W-w+2)-C\bigr)$.
\end{enumerate}
\end{invariant-chains*}

\noindent The theorem above is also proved in \cite{murai}, by studying $\Tor$ instead of $\Ext$ modules. A slight improvement in our work comes from the effective bound in part (2). As explained in Example~\ref{ex:regIn-sharp}, this bound is optimal.

\bigskip

All the theorems discussed so far are shadows of a more refined result that describes in a very precise fashion the graded components of the modules $\Ext^j_S(S/I,S)$, for arbitrary $\S_n$-invariant monomial ideals~$I$. This is the main result of the paper, and it follows closely the corresponding statement in the case of determinantal ideals \cite[Theorem~3.2]{raicu-reg-coh}. In particular, we not only describe the individual $\Ext$ modules, but also the natural maps induced by inclusions $I\supseteq J$, so one can for instance derive formulas for all the modules $\Ext^j_S(I/J,S)$. We will formulate our results here in a way that parallels those of \cite{raicu-reg-coh}.

We note that a monomial ideal $I\subseteq S$ is the same as one that is preserved by the natural action of the $n$-dimensional torus $(\kk^\times)^n$ on $S$ by rescaling the coordinates. If we consider the semi-direct product 
\begin{equation}\label{eq:def-G}
\G = (\kk^\times)^n \rtimes \S_n,
\end{equation}
where $\S_n$ acts on $(\kk^\times)^n$ by permuting the factors (also known as the \defi{wreath product} $\kk^\times \wr \S_n$), then an $\S_n$-invariant monomial ideal in $S$ is precisely the same as a $\G$-invariant ideal in $S$.

\begin{main-thm*}
 To any $\G$-invariant ideal $I\subseteq S$ we can associate a finite set $\mc{M}(I)$ of $\G$-equivariant $S$-modules with the property that for each $j\geq 0$
 \[\Ext^j_S(S/I,S) \simeq \bigoplus_{M\in\mc{M}(I)} \Ext^j_S(M,S),\]
 where the above isomorphism is $\G$-equivariant and degree preserving (but in general it does not preserve the $S$-module structure). In particular, we get
 \[\reg(S/I) = \max_{M\in\mc{M}(I)}\reg(M).\]
 The sets $\mc{M}(I)$ and the modules $\Ext^j_S(M,S)$ for $M\in\mc{M}(I)$ can be computed explicitly. Furthermore, the association $I\mapsto\mc{M}(I)$ has the property that whenever  $I\supseteq J$ are $\G$-invariant ideals, the (co)kernels and images of the induced maps $\Ext^j_S(S/I,S)\lra\Ext^j_S(S/J,S)$ can be computed as follows.
  \[\ker\left(\Ext^j_S(S/I,S)\lra\Ext^j_S(S/J,S)\right) = \bigoplus_{M\in\mc{M}(I)\setminus\mc{M}(J)}\Ext^{j}_S(M,S),\]
  \[\operatorname{Im}\left(\Ext^j_S(S/I,S)\lra\Ext^j_S(S/J,S)\right) = \bigoplus_{M\in\mc{M}(I)\cap\mc{M}(J)}\Ext^{j}_S(M,S),\]
  \[\coker\left(\Ext^j_S(S/I,S)\lra\Ext^j_S(S/J,S)\right) = \bigoplus_{M\in\mc{M}(J)\setminus\mc{M}(I)}\Ext^{j}_S(M,S).\]
 Finally, if we write $I:I_p^{\infty}$ for the saturation of $I$ with respect to $I_p$ then $\mc{M}(I:I_p^{\infty})\subseteq\mc{M}(I)$. More precisely
 \[\mc{M}(I:I_p^{\infty}) = \{M\in\mc{M}(I): \Ann(M) \not\subseteq I_p\}.\]
\end{main-thm*}

The precise statement of the Main Theorem is given in Theorem~\ref{thm:Ext-split-IX}. It has the remarkable consequence, pointed out by Satoshi Murai, that all $\S_n$-invariant monomial ideals are \defi{sequentially Cohen--Macaulay} (see \cite[Definition III.2.9]{stanley}, and Section~\ref{subsec:seq-CM}). The final theorem below characterizes the $\S_n$-invariant monomial ideals that are Cohen--Macaulay.

A famous question of Eisenbud--Musta\c t\uuu a--Stillman \cite[Question~6.2]{EMS} asks under what circumstances are the natural maps $\Ext^j_S(S/I,S) \lra H_I^j(S)$ injective. As explained in \cite[Example~6.3]{EMS}, a necessary condition is that the ideal $I$ is unmixed. For $\S_n$-invariant monomial ideals, we show that this condition is also sufficient, and that it is further equivalent to asking that the quotient $S/I$ is Cohen--Macaulay. We further characterize combinatorially those ideals for which these equivalent properties hold, as follows.

\begin{injectivity-Ext*}
 Let $\X$ be a set of pairwise incomparable partitions, and consider the corresponding ideal $I=I_{\X}\subseteq S$. The following are equivalent:
 \begin{enumerate}
  \item The natural maps $\Ext^j_S(S/I,S) \lra H_I^j(S)$ are injective for all $j$.
  \item $I$ is unmixed.
  \item Every partition $\x\in\X$ satisfies $x_1=\cdots=x_p$, where $p=\dim(S/I)+1$.
  \item For each $(\z,l)\in\Z(\X)$ one has $l=\dim(S/I)$.
  \item $S/I$ is Cohen--Macaulay.
 \end{enumerate}
\end{injectivity-Ext*}

\noindent As remarked in the proof of \cite[Corollary 3.8]{murai}, $\dim(S/I)=p-1$ where $p$ is the minimal number of parts of a partition in $\X$, that is, $p=\min\{ x'_1 : \x\in\X\}$.

\medskip

\noindent{\bf Organization.} In Section~\ref{sec:prelim} we establish basic facts about $\S_n$-invariant monomial ideals, discuss combinatorial aspects of Definition~\ref{def:ZX}, and study the $G$-equivariant $S$-modules that occur in the Main Theorem. In Section~\ref{sec:Ext-split-thick} we verify the Main Theorem, and deduce from there the Theorem on Regularity and Projective Dimension, as well as the Theorem on Injectivity of Maps from Ext to Local Cohomology. Section~\ref{sec:linear} is concerned with characterizing ideals with a linear free resolution, while Section~\ref{sec:reg-pows} discusses the explicit formula for the linear function computing regularity of powers of an ideal generated by the $\S_n$-orbit of a monomial. We end with Section~\ref{sec:vary-n} where we establish the Theorem on Invariant Chains of Ideals.

\section{Preliminaries}\label{sec:prelim}

The goal of this section is to introduce the main objects that are needed for the precise statement and the proof of the Main Theorem. In Section~\ref{subsec:GLinv-ideals} we discuss basic properties of $\S_n$-invariant monomial ideals, and discuss a number of important combinatorial implications of Definition~\ref{def:ZX}. In Section~\ref{subsec:subquots-Jzl} we introduce the $G$-equivariant $S$-modules that make up the sets $\mc{M}(I)$ in the Main Theorem, and for each such module $M$ and for $j\geq 0$ we compute $\Ext^j_S(M,S)$, and deduce from this calculation the projective dimension and Castelnuovo--Mumford regularity of $M$.

\subsection{$\G$-invariant ideals in $S$}\label{subsec:GLinv-ideals}

We let $S=\kk[e_1,\cdots,e_n]$ and let $\G$ be as in (\ref{eq:def-G}). We write $\langle V \rangle_{\kk}$ for the $\kk$-linear span of a collection $V$ of polynomials in $S$. For a partition $\x\in\P_n$ we let
\[S_{\x} = \langle \s(e^{\x}) : \s\in\S_n \rangle_{\kk}\]
which is an irreducible $G$-representation. The ring $S$ has a multiplicity-free decomposition
\[ S = \bigoplus_{\x\in\P_n} S_{\x}\]
into irreducible $G$-representations. Recall the definition of the ideals $I_{\x}$ in (\ref{eq:def-Ix}), generated by the component $S_{\x}$ in the decomposition above. We have
\begin{equation}\label{eq:dec-Ix}
 I_{\x} = \bigoplus_{\y\geq\x} S_{\y}.
\end{equation}

If we write $\sup(\x,\y)$ for the partition defined via
\begin{equation}\label{eq:supxy}
 \sup(\x,\y)_i = \max(x_i,y_i)
\end{equation}
then it follows from (\ref{eq:dec-Ix}) that
\begin{equation}\label{eq:IXcapIY}
 I_{\X} \cap I_{\Y} = \sum_{\x\in\X,\y\in\Y} I_{\sup(\x,\y)}.
\end{equation}

For every $\S_n$-invariant monomial ideal $I$ there exists a canonical choice of a subset $\X(I)\subset\P_n$ such that $I = I_{\X(I)}$. The set $\X(I)$ consists of the minimal partitions $\x\in\P_n$ (with respect to the order $\geq$), such that $e^{\x}\in I$. Up to the action of $\S_n$, the elements of $\X(I)$ give the minimal set of monomial generators of $I$.

\begin{remark}\label{rem:lattice-inv-ideals}
 The classification of $\G$-invariant ideals in $S$, together with (\ref{eq:IXcapIY}), shows that the lattice of $\G$-invariant ideals in $S$ is ismorphic to that of $\GL_m(\bb{C}) \times \GL_n(\bb{C})$-equivariant ideals in $\Sym(\bb{C}^m \oo \bb{C}^n)$, which was studied in \cite{deconcini-eisenbud-procesi}. Under this correspondence, the square-free ideals $I_p\subseteq S$ correspond to the determinantal ideals of $p\times p$ minors of the generic $m\times n$ matrix. Moreover, all the filtrations constructed in \cite{raicu-reg-coh} from chains of invariant ideals have corresponding analogues in the current setting. We will show that, just as in the case of matrices, these filtrations exhibit nice homological properties. 
\end{remark}

For $l=0,\cdots,n-1$ and $\z\in\P_n$, we consider the collection of partitions obtained from $\z$ by adding a single box to its Young diagram in row $(l+1)$ or higher (see \cite[Section~2B]{raicu-weyman} or \cite[Section~2.1]{raicu-reg-coh} for analogous definitions in the case of determinantal rings)
\begin{equation}\label{eq:defSucc}
 \mf{succ}(\z,l) = \{\x\in\P_n:\x\geq\z\rm{ and }x_i>z_i\rm{ for some }i>l\},\mbox{ and let}
\end{equation}
\begin{equation}\label{eq:defJzl}
 J_{\z,l} = I_{\z}/I_{\mf{succ}(\z,l)}.
\end{equation}

\begin{example}\label{ex:J-empty-l}
 If $\z=\emptyset$ then $I_{\z}=S$ and $I_{\mf{succ}(\z,l)}=I_{l+1}$, so $J_{\z,l}=S/I_{l+1}$ is the coordinate ring of the union of all the coordinate planes of dimension $l$.
\end{example}

To every $(\z,l)$ with $z_1=\cdots=z_{l+1}$, we associate the collection of rectangular partitions
\begin{equation}\label{eq:def-Yzl}
 \Y_{\z,l} = \left\{ \bigl((z_1+1)^{l+1}\bigr) \right\} \cup \left\{ \bigl((z_i+1)^i\bigr) : i>l+1\rm{ and }z_{i-1}>z_i\right\}.
\end{equation}

\begin{remark}\label{rem:z1--zl+1}
 As explained in \cite[Remark~3.4]{raicu-reg-coh}, it follows from Definition~\ref{def:ZX} that the condition $z_1=\cdots=z_{l+1}$ is automatically satisfied when $(\z,l)\in\Z(\X)$. Moreover, this condition is equivalent to the fact that the Young diagram of $\z$ has columns of size at least $l+1$, that is, $z'_j \geq l+1$ for all $j\leq z_1$.
\end{remark}

Unless otherwise specified, we will only consider pairs $(\z,l)$ with $z_1=\cdots=z_{l+1}$. We have the following.

\begin{lemma}\label{lem:zl-in-ZY-Yzl}
 If $(\z,l)\in\Z(\Y)$ then $I_{\Y} \subseteq I_{\Y_{\z,l}}$.
\end{lemma}

\begin{proof}
 Following Definition~\ref{def:ZX}, we set $c=z_1$ and note that for every $\y\in\Y$ we either have that $\z\not\geq\y(c)$, or $\z\geq\y(c)$ and $y'_{c+1}\geq l+1$. The latter case implies that $\y\geq((z_1+1)^{l+1})$, so $I_{\y} \subseteq I_{\Y_{\z,l}}$ by (\ref{eq:def-Yzl}) and (\ref{eq:dec-Ix}). In the former case, we must have $\min(y_i,c)>z_i$ for some $i$. Taking the smallest such $i$, we have that $z_{i-1}>z_i$, and by Remark~\ref{rem:z1--zl+1} we get that $i>l+1$. It follows that $\y\geq((z_i+1)^i)$, so $I_{\y} \subseteq I_{\Y_{\z,l}}$ as before.
\end{proof}

To obtain a characterization of the condition $(\z,l)\in\Z(\Y)$, we consider the subset $\Y'_{\z,l}\subset\Y_{\z,l}$ defined by
\begin{equation}\label{eq:def-Y'zl}
 \Y'_{\z,l} = \{ ((z_i+1)^i) : i>l+1\rm{ and }z_{i-1}>z_i\}.
\end{equation}

\begin{lemma}\label{lem:characterize-zl-in-ZY}
  We have that $(\z,l)\in\Z(\Y)$ if and only if $I_{\Y} \subseteq I_{\Y_{\z,l}}$ and $I_{\Y}\not\subseteq I_{\Y'_{\z,l}}$.
\end{lemma}

\begin{proof}
 Using Lemma~\ref{lem:zl-in-ZY-Yzl}, we may assume that $I_{\Y} \subseteq I_{\Y_{\z,l}}$. Under this hypothesis, we need to check that $(\z,l)\in\Z(\Y)$ if and only if $I_{\Y}\not\subseteq I_{\Y'_{\z,l}}$. We let $c=z_1$ as usual.
 
 Assume first that $(\z,l)\in\Z(\Y)$, so that there exists $\y\in\Y$ with $\z\geq\y(c)$. Suppose by contradiction that $I_{\Y}\subseteq I_{\Y'_{\z,l}}$. It follows from (\ref{eq:def-Y'zl}) that $\y\geq \bigl((z_i+1)^i\bigr)$ for some $i>l+1$ with $z_{i-1}>z_i$, which implies that $y_i>z_i$ and $c\geq z_{i-1}>z_i$. This shows that $\min(y_i,c) > z_i$, contradicting the assumption $\z\geq\y(c)$.

 Assume now $(\z,l)\not\in\Z(\Y)$, and suppose by contradiction that $I_{\Y}\not\subseteq I_{\Y'_{\z,l}}$. There exists then $\y\in\Y$ such that for every $i>l+1$ with $z_{i-1}>z_i$ we have $\y\not\geq\bigl((z_i+1)^i\bigr)$, which implies that $y_i\leq z_i$ for $i>l+1$. Using our standing assumption $z_1=\cdots=z_{l+1}$, it follows that $\z\geq\y(c)$. Moreover, since $I_{\y} \subseteq I_{\Y} \subseteq I_{\Y_{\z,l}}$, we must have $\y\geq\bigl((z_1+1)^{l+1}\bigr)$, that is $y'_{c+1}\geq l+1$. Since $y_i\leq z_i\leq c$ for $i>l+1$, we have in fact that $y'_{c+1}=l+1$, so $\y$ satisfies condition (1) in Definition~\ref{def:ZX}. Since $(\z,l)\not\in\Z(\Y)$, condition (2) must fail, that is, there exists $\x\in\Y$ with $\x(c)\leq\z$ and $x'_{c+1}\leq l$. This contradicts $I_{\Y} \subseteq I_{\Y_{\z,l}}$, since it implies that $I_{\x} \not\subseteq I_{\Y_{\z,l}}$.
\end{proof}

We also record the following direct consequence of \cite[Corollary~2.3]{raicu-reg-coh} and Remark~\ref{rem:lattice-inv-ideals}.

\begin{lemma}\label{lem:Jzl-sub-IY}
 There exists a $\G$-equivariant inclusion of $S$-modules $J_{\z,l}\subseteq S/I_{\Y}$ if and only if $I_{\Y_{\z,l}}\supseteq I_{\Y}\supseteq I_{\mf{succ}(\z,l)}$. Moreover, such an inclusion is uniquely defined up to a scalar, and gives rise to an exact sequence
  \begin{equation}\label{eq:ses-Jzl-Yzl}
  0\lra J_{\z,l} \lra \frac{S}{I_{\Y_{\z,l}}} \lra \frac{S}{I_{\z}+I_{\Y_{\z,l}}} \lra 0.
  \end{equation}
\end{lemma}

\begin{proposition}\label{prop:ZYcupz}
 If there exists a $\G$-equivariant inclusion $J_{\z,l} \subseteq S/I_{\Y}$ and $(\z,l) \in \Z(\Y)$ then
 \[ \Z(\Y \cup \{\z\}) = \Z(\Y) \setminus \{ (\z,l) \}.\]
\end{proposition}

Before going into the details of the proof, it may be worthwhile to analyze one example.

\begin{example}\label{ex:ZYcupz}
 Let $n=3$ and $\Y = \{(2,1,1),(4,2)\}$, which was denoted $\X$ in Example~\ref{ex:reg-pdim}, and recall that
 \[ \Z(\Y) =  \left\{ \left(\emptyset,1\right),\left(\ydiagram{1,1},1\right),\left(\ydiagram{1,1,1},0\right),\left(\ydiagram{2,2},0\right),\left(\ydiagram{3,2},0\right),\left(\ydiagram{3,3},0\right)\right\}.\]
Consider the pair $(\z,l) = \left(\ydiagram{1,1,1},0\right)$. We have that $\Y_{\z,l} = \{\ydiagram{2}\}$ and $I_{\mf{succ}(\z,l)} = I_{(2,1,1)}$, so the conditions in Lemma~\ref{lem:Jzl-sub-IY} are satisfied, showing that $J_{\z,l}\subseteq S/I_{\Y}$. One can check using Definition~\ref{def:ZX} that
\[ \Z(\Y \cup \{\z\}) = \left\{ \left(\emptyset,1\right),\left(\ydiagram{1,1},1\right),\left(\ydiagram{2,2},0\right),\left(\ydiagram{3,2},0\right),\left(\ydiagram{3,3},0\right)\right\} = \Z(\Y) \setminus \{ (\z,l) \},\]
confirming Proposition~\ref{prop:ZYcupz} in this special case. The reader can check that a similar conclusion holds when $(\z,l)=\left(\ydiagram{3,3},0\right)$, in which case $\Y_{\z,l} = \left\{\ydiagram{4},\ydiagram{1,1,1}\right\}$ and $I_{\mf{succ}(\z,l)} = I_{(4,3)} + I_{(3,3,1)}$.

Suppose now that $(\z,l)=\left(\ydiagram{2,2},0\right)$, so that $\Y_{\z,l} = \left\{\ydiagram{3},\ydiagram{1,1,1}\right\}$ and $I_{\mf{succ}(\z,l)} = I_{(3,2)} + I_{(2,2,1)}$. The inclusion $I_{\Y_{\z,l}}\supseteq I_{\Y}$ still holds, but $I_{\Y}\supseteq I_{\mf{succ}(\z,l)}$ fails, so $J_{\z,l}\not\subseteq S/I_{\Y}$ by Lemma~\ref{lem:Jzl-sub-IY}. One can check that
\[ \Z(\Y \cup \{\z\}) = \left\{ \left(\emptyset,1\right),\left(\ydiagram{1,1},1\right),\left(\ydiagram{1,1,1},0\right) \right\},\]
which is strictly smaller than $\Z(\Y) \setminus \{ (\z,l) \}$.
\end{example}

\begin{proof}[Proof of Proposition~\ref{prop:ZYcupz}]
 We set $\X = \Y \cup \{\z\}$ and note that $(\z,l)\not\in\Z(\X)$: indeed, if we take $\x=\z\in\X$ and let $c=z_1$ then $\z\geq\x(c)$ and $x'_{c+1}=0 < l+1$, so condition (2) in Definition~\ref{def:ZX} fails for $(\z,l)$.
 
 We next prove the inclusion $\Z(\X) \subset \Z(\Y)$. We suppose by contradiction that this isn't the case, consider $(\y,u) \in \Z(\X) \setminus \Z(\Y)$, and let $d=y_1$. Since $\Y\subset\X$, it follows that $(\y,u)$ satisfies condition (2) in the definition of $\Z(\Y)$. It must therefore fail condition (1), so there is no $\x\in\Y$ with $\y\geq\x(d)$ and $x'_{d+1}\leq u+1$. Since $\X\setminus\Y = \{\z\}$ and $(\y,u)\in\Z(\X)$, it follows that $\y\geq\z(d)$ and $z'_{d+1}=u+1$. Since $z'_{d+1}$ is non-zero, we conclude by Remark~\ref{rem:z1--zl+1} that $z'_{d+1}\geq l+1$ and $d+1\leq c$, or equivalently, $u\geq l$ and $d<c$. Since $(\z,l)\in\Z(\Y)$, we can find $\x\in\Y$ with $\z\geq\x(c)$ and $x'_{c+1}=l+1$. Since $d<c$, this implies $x'_{d+1}\leq z'_{d+1}=u+1$, and $\y\geq\z(d)\geq\x(d)$, so $(\y,u)$ satisfies condition (1) in the definition of $\Z(\Y)$, which we saw was impossible.
 
 To conclude our proof, we have to check that every $(\y,u)\in\Z(\Y) \setminus \{(\z,l)\}$ belongs to $\Z(\X)$. Since $\Y\subset\X$ and $(\y,u)\in\Z(\Y)$, it follows that $(\y,u)$ satisfies condition (1) in the definition of $\Z(\X)$. We let $d=y_1$ as before. If we assume by contradiction that $(\y,u)\not\in\Z(\X)$, then it must fail condition (2). Since $\X\setminus\Y = \{\z\}$, the only way this can happen is if $\y\geq\z(d)$ and $z'_{d+1}\leq u$. 
 
 Suppose first that $d<c$. Since $(\z,l)\in\Z(\Y)$, we can find $\x\in\Y$ with $\z\geq\x(c)$, so $\y\geq\z(d) \geq \x(d)$, and $x'_{d+1}\leq z'_{d+1}\leq u$. Since $\x\in\X$, this means that $(\y,u)$ fails condition (2) in the definition of $\Z(\X)$, a contradiction.
 
 Suppose now that $d\geq c$, so that $\y\geq\z(d)=\z$. Since $z_1=c$, we also have $\y(c)\geq\z$. If $\y(c) \neq \z$ then there exists $\t\in\mf{succ}(\z,l)$ such that $\y(c)\geq\t$. The assumption $J_{\z,l} \subseteq S/I_{\Y}$ combined with Lemma~\ref{lem:Jzl-sub-IY} implies that $I_{\mf{succ}(\z,l)}\subseteq I_{\Y}$, so there exists $\x\in\Y$ with $\t\geq\x$. Since $d\geq c\geq t_1\geq x_1$, we get $\x=\x(d)$ and thus $x'_{d+1}=0$; moreover, since $\y\geq\t\geq\x$, it follows that $\y\geq\x(d)$, so $(\y,u)$ fails condition (2) in the definition of $\Z(\Y)$, a contradiction. We can therefore assume that $\y(c)=\z$. If $c=d$ then $\y=\z$ and therefore $(\y,u)=(\z,l)$ by \cite[Remark~3.5]{raicu-reg-coh}, so we may assume that $c<d$. If $y'_{c+1}\leq l$, choose any $\x\in \Y$ with $\y\geq\x(d)$ and note that $\z=\y(c)\geq\x(c)$ and $x'_{c+1}\leq y'_{c+1}\leq l$, that is $(\z,l)$ fails property (2) in the definition of $\Z(\Y)$, a contradiction. We may therefore assume that $y'_{c+1}\geq l+1$, so $\y\in\mf{succ}(\z,l)$. Since $I_{\mf{succ}(\z,l)}\subseteq I_{\Y}$, we can then find $\x\in\Y$ with $\y\geq\x\geq\x(d)$ which then satisfies $x'_{d+1}\leq y'_{d+1}=0<u+1$, so $(\y,u)$ fails property (2) in the definition of $\Z(\Y)$, again a contradiction. 
\end{proof}

The following results will be used in Section~\ref{sec:Ext-split-thick}. 

\begin{lemma}\label{lem:yu-in-ZY}
 Suppose that $\Y\subset\P_n$ is a subset with the property that for all $\x\in\Y$ and all $j\geq 0$, either $x'_j=0$ or $x'_j\geq l+1$. For every $(\y,u)\in\Z(\Y)$ we have $u\geq l$.
\end{lemma}

\begin{proof}
 Suppose that $(\y,u)\in\Z(\Y)$ and let $d=y_1$. It follows from condition (2) in Definition~\ref{def:ZX} that there exists an element $\x\in\Y$ with $x'_{d+1} = u+1$. Taking $j=d+1$, we have that $x'_j\neq 0$, so $u+1=x'_j \geq l+1$. This proves that $u\geq l$, as desired.
\end{proof}

\begin{corollary}\label{cor:conn-hom-Yzl-0}
 If $(\y,u)\in\Z(\Y_{\z,l} \cup \{\z\})$ then $u\geq l$ (and in particular $u\neq l-1$).
\end{corollary}

\begin{proof}
We apply Lemma~\ref{lem:yu-in-ZY} with $\Y=\Y_{\z,l} \cup \{\z\}$: to check the hypothesis, we choose $\x\in\Y_{\z,l} \cup \{\z\}$. If $\x=\z$ then we have by Remark~\ref{rem:z1--zl+1} that $x'_{j}\geq l+1$ whenever $x'_j\neq 0$. If $\x=((z_1+1)^{l+1})$ then $x'_j=l+1$ whenever $x'_j\neq 0$. If $\x=((z_i+1)^i)$ then $x'_j=i>l+1$ whenever $x'_j\neq 0$.
\end{proof}

Recall the definition of the \defi{saturation} of an ideal $I$ with respect to $J$,
\begin{equation}\label{eq:def-saturation}
I:J^{\infty} = \{f\in S: f\cdot J^d \subseteq I\textrm{ for }d\gg 0\}.
\end{equation}
When $I=I_{\X}$ and $J=I_p$ the saturation can be described concretely as follows. For $\X\subset\P_n$ we define 
\begin{equation}\label{eq:p-saturation}
 \X^{:p} = \{\x(c): \x\in\X,c\in\bb{Z}_{\geq 0}, x'_c> p\rm{ if }c>0,\rm{ and }x'_{c+1}\leq p\}
\end{equation}
In terms of Young diagrams, we can think of $\X^{:p}$ as being obtained from $\X$ by removing from each $\x\in\X$ the columns of size $\leq p$. Using \cite[Lemma~2.3]{raicu-reg-coh} and its proof, we obtain
\begin{equation}\label{eq:saturation}
I_{\X}:I_p^{\infty} = I_{\X^{:p}}.
\end{equation}
In analogy with \cite[Corollary~2.4]{raicu-reg-coh}, we also have that
\begin{equation}\label{eq:Ann-Jzl}
 \Ann(J_{\z,l}) = I_{l+1},
\end{equation}
that is, the scheme-theoretic support of $J_{\z,l}$ is the union of the coordinate planes of dimension~$l$ in $\kk^n$ (see also Example~\ref{ex:J-empty-l}, and Proposition~\ref{prop:decomp-Jzl} below).

\subsection{$\Ext$ modules for the subquotients $J_{\z,l}$}\label{subsec:subquots-Jzl}

The goal of this section is to prove that $J_{\z,l}$ is Cohen--Macaulay, to compute its regularity and projective dimension, and to describe the modules $\Ext^j_S(J_{\z,l},S)$. We will obtain this from a natural decomposition of $J_{\z,l}$ into a direct sum of cyclic modules, which we describe next. If $r\geq 1$, we write $[r]$ for the set $\{1,\cdots,r\}$. Given $\LL\subset [n]$ we consider the cyclic module
\[ S_{\LL} = \kk[e_i:i\in\LL] \simeq \frac{S}{\langle e_i:i\notin\LL\rangle},\]
where we abuse notation and write $e_i$ for the equivalence class of $e_i\in S$ in the quotient $S_{\LL}$. We define the ideals $I_p^{\LL}\subseteq S_{\LL}$ as the image of the square-free ideals $I_p\subseteq S$ via the quotient map, for $p=1,\cdots,|\LL|$. We write $\S_{\LL}$ for the subgroup of $\S_n$ consisting of permutations $\s$ that fix every element $i\in[n]\setminus\LL$, and note that $\S_{\LL}$ is isomorphic to the group of permutations of the set $\LL$.

We will be working with finitely generated $\bb{Z}^n$-graded $S$-modules $M$, and for $\u\in\bb{Z}^n$ we will write $M_{\u}$ for the \defi{$\u$-graded} (or \defi{$\u$-isotypic}) component of $M$, which is a finite dimensional vector space: we call its dimension $\dim(M_{\u})$ the \defi{multiplicity} of $\u$ in $M$. If $\v\in\bb{Z}^n$, we let $M(-\v)$ denote the \defi{shifted} $\bb{Z}^n$-graded $S$-module with $M(-\v)_{\u} = M_{-\v+\u}$. In order to be able to refer to elements of shifted modules, it will often be convenient to write $S(-\v) = S\cdot E_{\v}$, where $E_{\v}$ is a generator of the free module $S(-\v)$, and thus it has degree $\deg(E_{\v})=\v$. Using the identification $M(-\v) = M \oo_S S(-\v)$, we will write more generally $M(-\v) = M\cdot E_{\v}$, so that if $m\in M$ has degree $\u$, then $m\cdot E_{\v}\in M(-\v)$ has degree $\u+\v$.

We write $\mc{O}(\z)$ for the orbit of the $\S_n$-action on $\bb{Z}^n$ of some element $\z$. We set $z_{n+1}=-\infty$ and consider
\[ 0\leq l<p\leq n,\mbox{ and a partition }\z\in\P_n\mbox{ with }z_1=\cdots=z_p>z_{p+1}.\]
We let $c=z_1$, and for each $\u\in\mc{O}(\z)$ we define the set
\[ \LL_{\u} = \{ j\in[n] : u_j = c\}.\]

\begin{proposition}\label{prop:decomp-Jzl}
 With the notation above, we have an isomorphism of $G$-equivariant $S$-modules
 \[ J_{\z,l} \simeq \bigoplus_{\u\in\mc{O}(\z)} \frac{S_{\LL_{\u}}}{I_{l+1}^{\LL_{\u}}} \cdot E_{\u}.\]
\end{proposition}

\begin{proof} Since $J_{\z,l}$ is a quotient of $I_{\z}$, it follows that we have a natural $G$-equivariant surjection
\[ \pi : \bigoplus_{\u\in\mc{O}(\z)} S\cdot E_{\u} \lra J_{\z,l},\]
sending $E_{\u}$ to the residue class of the monomial $e^{\u}$. We have that $\ker(\pi)$ contains $e_1\cdots e_{l+1} \cdot E_{\z}$, as well as $e_i\cdot E_{\z}$ for all $i=p+1,\cdots,n$, since the monomials $e_1\cdots e_{l+1}\cdot e^{\z}$ and $e_i\cdot e^{\z}$ are in $I_{\mf{succ}(\z,l)}$. Using the fact that $e^{\z}$ is preserved by permutations $\s\in\S_{[p]}$, we find that $e^{\z} \cdot (I_{l+1} + \langle e_{p+1},\cdots,e_n\rangle)$ is contained in $\ker(\pi)$. Using the $\S_n$-action we conclude that $\pi$ induces a surjection
\begin{equation}\label{eq:onto-Jzl}
 \bigoplus_{\u\in\mc{O}(\z)} \frac{S_{\LL_{\u}}}{I_{l+1}^{\LL_{\u}}} \cdot E_{\u}\lra J_{\z,l}.
\end{equation}
To prove that this is an isomorphism, it suffices by $\S_n$-equivariance to check that for every dominant $\x$, the $\x$-isotypic components in the source and target have the same multiplicity. Using (\ref{eq:defSucc})--(\ref{eq:defJzl}), the $\x$-isoptypic component of $J_{\z,l}$ is non-zero precisely when
\begin{equation}\label{eq:rel-x-to-z}
x_i = z_i\mbox{ for }i\geq l+1,\mbox{ and }x_i\geq c\mbox{ for }i=1,\cdots,l,
\end{equation}
in which case it has multiplicity one. Suppose now that the $\x$-isotypic component of the source of (\ref{eq:onto-Jzl}) is non-zero, so that we can write $\x =\u + \v$ for some $\u\in\mc{O}(\z)$ and $\v$ such that the $\v$-isotypic component of $S_{\LL_{\u}}/I_{l+1}^{\LL_{\u}}$ is non-zero. Since $u_i=c$ if and only if $i\in\LL_{\u}$, and $v_i=0$ for $i\not\in\LL_{\u}$, we conclude that
\[ x_i=u_i+v_i \geq c \mbox{ for }i\in\LL_{\u},\mbox{ and }x_i=u_i+v_i=u_i < c \mbox{ for }i\not\in\LL_{\u}.\]
Since $\x$ is dominant, we conclude that $\LL_{\u} = [p]$ and that $\u,\v$ must be dominant as well. This implies further that $\u=\z$, and that $\v$ is a dominant weight of $S/(I_{l+1} + \langle e_{p+1},\cdots,e_n\rangle)$. We conclude that at most $l$ of the entries of $\v$ are non-zero, so that $v_i=0$ for $i\geq l+1$. Therefore, $\x=\u+\v$ satisfies (\ref{eq:rel-x-to-z}) and has multiplicity one in the representation on the left of (\ref{eq:onto-Jzl}), proving that (\ref{eq:onto-Jzl}) is an isomorphism.
\end{proof}

We introduce one more piece of notation: for a finitely generated $\bb{Z}^n$-graded $S$-module $M$, we define its \defi{character $[M]$} to be the Laurent power series
\[ [M] = \sum_{\u\in\bb{Z}^n} \dim(M_{\u}) \cdot e^{\u} \in \bb{Z}((e_1,\cdots,e_n)).\]
Note the abuse of notation where we use the same symbols $e_i$ as for the variables in $S$, but this shouldn't cause any confusion, but rather make the notation more intuitive! Note also that we have 
\[[S]=\sum_{\u\in\bb{Z}^n_{\geq 0}} e^{\u},\quad\mbox{ and }\quad[M(-\v)] = [M\cdot E_{\v}] = [M]\cdot e^{\v}.\]

Using Proposition~\ref{prop:decomp-Jzl}, we can now describe the $\Ext$ modules of the subquotients $J_{\z,l}$. We start with the special case when $\z=\emptyset$, which corresponds to $J_{\z,l} = S/I_{l+1}$ (see Example~\ref{ex:J-empty-l}). For a tuple $\t\in\bb{Z}^n_{\geq 0}$ we write $\op{supp}(\t)=\{i\in[n]: t_i\neq 0\}$ for the support of $\t$, and let $p_{\t}=|\op{supp}(\t)|$ denote its cardinality. 

\begin{lemma}\label{lem:char-J0l}
We have that $\Ext^j(S/I_{l+1},S) = 0$ for $j\neq n-l$ and
\[\left[\Ext^{n-l}_S(S/I_{l+1},S)\right] = \sum_{\substack{\t\in\bb{Z}^n_{\geq 0} \\ p_{\t} \leq l}} {n-1-p_{\t} \choose l - p_{\t}} \cdot e^{\t-(1^n)}.\]
\end{lemma}

\begin{proof} The vanishing of $\Ext^j(S/I_{l+1},S)$ for $j\neq n-l$ follows from the fact that $S/I_{l+1}$ is Cohen--Macaulay of dimension $l$. If we write $\omega_S$ for the canonical module of $S$ then $\omega_S \simeq S\cdot E_{(1^n)}$, so we have to show that
\[\left[\Ext^{n-l}_S(S/I_{l+1},\omega_S)\right] = \sum_{\substack{\t\in\bb{Z}^n_{\geq 0} \\ p_{\t} \leq l}} {n-1-p_{\t} \choose l - p_{\t}} \cdot e^{\t}.\]
It follows from \cite[Theorem~2.6]{yanagawa} or \cite[Theorem~3.3]{mustata} that $\Ext^{n-l}_S(S/I_{l+1},\omega_S)$ is a square-free $S$-module, that is, the multiplicity of any $\t$-isotypic component is non-zero only for $\t\in\bb{Z}^n_{\geq 0}$ and it depends only on $\op{supp}(\t)$. We may then assume that $\t$ has entries $0$ and $1$, with $p_{\t}$ of them equal to $1$. Using the fact that the Alexander dual of $I_{l+1}$ is $I_{n-l}$, it follows from \cite[Theorem~3.4]{yanagawa} that
\[ \dim\left(\Ext^{n-l}_S(S/I_{l+1},\omega_S)_{\t}\right) = \dim \left(\Tor^S_{l-p_{\t}}(I_{n-l},S)_{(1^n)-\t}\right),\]
and the right hand side is equal to ${n-1-p_{\t} \choose l - p_{\t}}$: this follows for instance from \cite[Theorem~4.11]{galetto} and the fact that the number of standard Young tableaux of hook-shape $(n-l,1^{l-p_{\t}})$ with entries in $[n]\setminus\op{supp}(\t)$ is equal by the Hook Length Formula \cite[Section~4.3]{fulton} to ${n-1-p_{\t} \choose l - p_{\t}}$.
\end{proof}

\begin{corollary}\label{cor:properties-Jzl}
 If $z_1=\cdots=z_{l+1}$ then the module $J_{\z,l}$ is Cohen--Macaulay of projective dimension $n-l$ and regularity $|\z|+l$. Moreover, we have that
 \[ \left[\Ext^{n-l}_S(J_{\z,l},S)\right] = \bigoplus_{\substack{ \u\in\mc{O}_{\z}  \\ \v\in\bb{Z}^{\LL_{\u}}_{\geq 0} \\ p_{\v}\leq l}}  {p-1-p_{\v} \choose l - p_{\v}} \cdot e^{\v-\u-(1^n)}.\]
\end{corollary}
 
\begin{proof} Using Proposition~\ref{prop:decomp-Jzl}, $J_{\z,l}$ is a direct sum of Cohen--Macaualy modules of dimension $l$, so it is itself Cohen--Macaualy and moreover the projective dimension is $\dim(S)-l=n-l$. The formula for the character of $\Ext^{n-l}_S(J_{\z,l},S)$ follows from Proposition~\ref{prop:decomp-Jzl} and Lemma~\ref{lem:char-J0l} (applied to $S_{\LL_{\u}}$) by observing that $\LL_{\u}$ is a set of cardinality $p$, so $S_{\LL_{\u}}$ is a polynomial ring of dimension $p$, and 
\[\Ext^{n-l}_S\left(\frac{S_{\LL_{\u}}}{I_{l+1}^{\LL_{\u}}} \cdot E_{\u},\omega_S\right) = \Ext^{n-l}_S\left(\frac{S_{\LL_{\u}}}{I_{l+1}^{\LL_{\u}}},\omega_S\right)  \cdot E_{-\u} = \Ext^{p-l}_{S_{\LL_{\u}}}\left(\frac{S_{\LL_{\u}}}{I_{l+1}^{\LL_{\u}}},\omega_{S_{\LL_{\u}}}\right)  \cdot E_{-\u}.\]
To compute the regularity, we use the fact that for a graded module $M$ we have
\begin{equation}\label{eq:regM}
\reg(M)=\max\{-r-j:\Ext^j_S(M,S)_r\neq 0\}.
\end{equation}
Applying this to $M=J_{\z,l}$ we see that the only cohomological degree $j$ where the $\Ext$-module is non-zero is when $j=n-l$, and the minimal degree $r$ for which $\Ext^{n-l}_S(J_{\z,l},S)_r\neq 0$ is attained when $\v=(0^n)$ and is equal to
\[ r = |\v| - |\u| - n = 0 - |\z| - n,\]
showing that the regularity of $J_{\z,l}$ is $\reg(J_{\z,l}) = |\z|+n-(n-l) = |\z|+l$, as desired.
\end{proof}
 
\section{$\Ext$ modules for $\S_n$-invariant monomial ideals}\label{sec:Ext-split-thick}

We let $S=\kk[e_1,\cdots,e_n]$ and let $\G$ be as in (\ref{eq:def-G}). The goal of this section is to prove the Main Theorem from the Introduction, describing the modules $\Ext^j_S(S/I,S)$ for $I$ a $\G$-invariant ideal in $S$, as well as the natural maps between these modules. Our arguments follow closely the strategy employed in the study of invariant ideals in the ring of polynomial functions on the space of matrices from \cite{raicu-reg-coh}. As a consequence of the Main Theorem, we deduce the Theorem on Regularity and Projective Dimension in Section~\ref{subsec:reg-projdim}, and the Theorem on Injectivity of Maps from Ext to Local Cohomology in Section~\ref{subsec:inj-loccoh}. In Section~\ref{subsec:seq-CM} we show that every $\G$-invariant ideal $I$ has the property that $S/I$ is sequentially Cohen--Macaulay.

Using Definition~\ref{def:ZX} for the set $\Z(\X)$, the Main Theorem can be formulated more precisely as follows.

\begin{theorem}\label{thm:Ext-split-IX}
 Let $\X\subseteq\P_n$ and let $I_{\X}\subseteq S$ denote the associated $\G$-invariant ideal. For each $j\geq 0$ there exists a degree preserving isomorphism of $G$-representations (but in general, not of $S$-modules)
 \begin{equation}\label{eq:ExtS/IX}
  \Ext^j_S(S/I_{\X},S) \simeq \bigoplus_{(\ul{z},l)\in\Z(\X)}\Ext^{j}_S(J_{\z,l},S).
 \end{equation}
Moreover, if $\X,\Y\subset\P_n$ are such that $I_{\X}\subseteq I_{\Y}$, then the natural surjection $S/I_{\X}\onto S/I_{\Y}$ induces maps $\Ext^j_S(S/I_{\Y},S)\lra\Ext^j_S(S/I_{\X},S)$ for all $j\geq 0$, whose (co)kernels and images can be described via
 \begin{equation}\label{eq:kerExt}
 \ker\left(\Ext^j_S(S/I_{\Y},S)\lra\Ext^j_S(S/I_{\X},S)\right) \simeq \bigoplus_{(\ul{z},l)\in\Z(\Y)\setminus\Z(\X)}\Ext^{j}_S(J_{\z,l},S),
 \end{equation}
 \begin{equation}\label{eq:imExt}
 \operatorname{Im}\left(\Ext^j_S(S/I_{\Y},S)\lra\Ext^j_S(S/I_{\X},S)\right) \simeq \bigoplus_{(\ul{z},l)\in\Z(\Y)\cap\Z(\X)}\Ext^{j}_S(J_{\z,l},S),
 \end{equation}
 \begin{equation}\label{eq:cokerExt}
 \coker\left(\Ext^j_S(S/I_{\Y},S)\lra\Ext^j_S(S/I_{\X},S)\right) \simeq \bigoplus_{(\ul{z},l)\in\Z(\X)\setminus\Z(\Y)}\Ext^{j}_S(J_{\z,l},S).
 \end{equation}
 Finally, recall that the saturation of $I_{\X}$ with respect to $I_p$ is given by $I_{\X^{:p}}$ (see (\ref{eq:p-saturation})--(\ref{eq:saturation})). We have
 \begin{equation}\label{eq:Z-saturation}
 \Z(\X^{:p}) = \{(\z,l)\in\Z(\X): l\geq p\} \subseteq \Z(\X).
 \end{equation}
 In particular, if we apply (\ref{eq:kerExt}) to the inclusion $I_{\X}\subseteq I_{\X^{:p}}$ we obtain for each $j\geq 0$ injective maps
 \[\Ext^j_S(S/I_{\X^{:p}},S) \lra \Ext^j_S(S/I_{\X},S).\]
\end{theorem}

The equality (\ref{eq:Z-saturation}) is the same as that in \cite[(3.6)]{raicu-reg-coh} and is proved there. To prove Theorem~\ref{thm:Ext-split-IX} we construct for each pair $(\z,l)$ and each subset $\X\subset\P_n$ a collection of $G$-submodules
\begin{equation}\label{eq:EjzlX}
 \E^j_{\z,l}(\X) \subseteq \Ext^j_S(S/I_{\X},S)
\end{equation}
with the following properties. 
\begin{enumerate}

\item For every inclusion $I_{\X}\subseteq I_{\Y}$, if we denote by $\pi:S/I_{\X}\onto S/I_{\Y}$ the corresponding quotient map, then for $j\geq 0$ the induced map
\begin{equation}\label{eq:pi^j-YtoX}
\pi^j_{\Y\to\X}:\Ext^j_S(S/I_{\Y},S) \lra \Ext^j_S(S/I_{\X},S)
\end{equation}
has the property that for all $j\geq 0$ and all pairs $(\z,l)$ such that $I_{\X} \subseteq I_{\Y} \subseteq I_{\Y_{\z,l}}$
\begin{equation}\label{eq:E^j-YtoX}
\E^j_{\z,l}(\X) = \pi^j_{\Y\to\X}(\E^j_{\z,l}(\Y)).
\end{equation}

\item For each $j\geq 0$, we have that
\begin{equation}\label{eq:E^j_zl-X}
 \E^j_{\z,l}(\X)\simeq\begin{cases}
  \Ext^{j}_S(J_{\z,l},S) & \rm{when }(\z,l)\in\Z(\X), \\
  0 & \rm{otherwise}.
 \end{cases}
\end{equation}

\item For each $\X\subset\P_n$ and $j\geq 0$, the inclusions (\ref{eq:EjzlX}) give rise to a decomposition
\begin{equation}\label{eq:Ext^j=sum-E^j}
\Ext^j_S(S/I_{\X},S) = \bigoplus_{(\z,l)\in\Z(\X)} \E^j_{\z,l}(\X).
\end{equation}
\end{enumerate}

Once these properties are established, (\ref{eq:ExtS/IX}) follows by combining (\ref{eq:Ext^j=sum-E^j}) with (\ref{eq:E^j_zl-X}), while (\ref{eq:kerExt})--(\ref{eq:cokerExt}) follow from (\ref{eq:E^j-YtoX}) and Lemma~\ref{lem:zl-in-ZY-Yzl}. To construct the $\G$-submodules in (\ref{eq:EjzlX}) we start with the following.

\begin{lemma}\label{lem:phi-jzl-onto}
 For every pair $(\z,l)$ where $\z\in\P_n$ is a partition satisfying $z_1=\cdots=z_{l+1}$, the short exact sequence (\ref{eq:ses-Jzl-Yzl}) induces $G$-equivariant surjective maps at the level of $\Ext$ modules
 \begin{equation}\label{eq:surj-phi^j_zl}
 \phi^j_{\z,l}:\Ext^j_S(S/I_{\Y_{\z,l}},S) \onto \Ext^j_S(J_{\z,l},S)\mbox{ for }j\geq 0.
 \end{equation}
\end{lemma}

\noindent To prove this result, we use the following analogue of \cite[Corollary~3.7]{raicu-reg-coh}, which follows from Remark~\ref{rem:lattice-inv-ideals}.

\begin{proposition}\label{prop:filtration-S/IX}
 There exists a $\G$-equivariant filtration of $S/I_{\X}$ whose successive quotients are the modules $J_{\z,l}$ with $(\z,l)\in\Z(\X)$. In particular (see \cite[Lemma~2.9]{raicu-reg-coh}), there exists a $G$-equivariant inclusion
 \[\Ext^j_S(S/I_{\X},S) \lra \bigoplus_{(\ul{z},l)\in\Z(\X)}\Ext^{j}_S(J_{\z,l},S)\mbox{  for each }j\geq 0.\]
\end{proposition}

\begin{proof}[Proof of Lemma~\ref{lem:phi-jzl-onto}]
 In light of Corollary~\ref{cor:properties-Jzl}, we only need to consider $j=n-l$, since $\Ext^j_S(J_{\z,l},S)=0$ for $j\neq n-l$. To check (\ref{eq:surj-phi^j_zl}) in this case, it suffices to prove that the connecting homomorphism
\begin{equation}\label{eq:conn-hom-Yzl-0}
\Ext^{n-l}_S(J_{\z,l},S) \lra \Ext^{n-l+1}_S(S/(I_{\z}+I_{\Y_{\z,l}}),S)
\end{equation}
is identically zero. Using Proposition~\ref{prop:filtration-S/IX} with $\X=\Y_{\z,l}\cup\{\z\}$, it suffices to check that
\[\bigoplus_{(\y,u)\in\Z(\Y_{\z,l}\cup\{\z\})}\Ext^{n-l+1}_S(J_{\y,u},S)=0.\]
Using Corollary~\ref{cor:properties-Jzl}, this follows from Corollary~\ref{cor:conn-hom-Yzl-0}, since no $(\y,u)\in\Z(\Y_{\z,l}\cup\{\z\})$ satisfies $u=l-1$.
\end{proof}

Using Lemma~\ref{lem:phi-jzl-onto}, we can choose a $\G$-submodule 
 \begin{equation}\label{eq:def-Ejzl}
 \E^j_{\z,l}\subseteq \Ext^j_S(S/I_{\Y_{\z,l}},S)
 \end{equation}
with the property that $\phi^j_{\z,l}$ maps $\E^j_{\z,l}$ isomorphically onto $\Ext^j_S(J_{\z,l},S)$, and note that by Corollary~\ref{cor:properties-Jzl}, $\E^j_{\z,l}=0$ if $j\neq n-l$. For every ideal $I_{\X} \subseteq I_{\Y_{\z,l}}$, we define
\[\E^j_{\z,l}(\X) = \pi^j_{\Y_{\z,l}\to\X}(\E^j_{\z,l}),\]
so that in particular $\E^j_{\z,l}(\Y_{\z,l})=\E^j_{\z,l}$. We also let $\E^j_{\z,l}(\X)=0$ if $I_{\X} \not\subseteq I_{\Y_{\z,l}}$. The equality (\ref{eq:E^j-YtoX}) follows from the functoriality of maps of $\Ext$ modules.

To prove (\ref{eq:E^j_zl-X}), we may assume that $I_{\X} \subseteq I_{\Y_{\z,l}}$: indeed, if that's not the case then $\E^j_{\z,l}(\X)=0$ by definition, and it follows from Lemma~\ref{lem:zl-in-ZY-Yzl} that $(\z,l)\not\in\Z(\X)$. We show that if $(\z,l)\not\in\Z(\X)$ then $\E^j_{\z,l}(\X)=0$, which again is interesting only for $j=n-l$. Using Lemma~\ref{lem:characterize-zl-in-ZY}, it follows that $I_{\X} \subseteq I_{\Y'_{\z,l}}$, so by (\ref{eq:E^j-YtoX}) it suffices to prove that $\E^{n-l}_{\z,l}(\Y'_{\z,l})=0$. For this, it is then enough to check that $\Ext^{n-l}_S(S/I_{\Y'_{\z,l}},S)=0$, which follows from Proposition~\ref{prop:filtration-S/IX} and Corollary~\ref{cor:properties-Jzl} if we can show that there is no $(\y,u)\in\Z(\Y'_{\z,l})$ with $u=l$. Using (\ref{eq:def-Y'zl}), every $\x\in\Y'_{\z,l}$ has the property that $x'_j>l+1$ whenever $x'_j\neq 0$, so Lemma~\ref{lem:yu-in-ZY} applies to give the desired conclusion.

To end the proof of Theorem~\ref{thm:Ext-split-IX}, we show simultaneously that $\E^j_{\z,l}(\X)\simeq\Ext^{j}_S(J_{\z,l},S)$ when $(\z,l)\in\Z(\X)$ and that (\ref{eq:Ext^j=sum-E^j}) holds, by induction on the size of $\Z(\X)$. The smallest non-zero term in the filtration from Proposition~\ref{prop:filtration-S/IX} gives a $G$-equivariant inclusion $J_{\z,l}\subseteq S/I_{\X}$, with $(\z,l)\in\Z(\X)$. It follows from Lemma~\ref{lem:Jzl-sub-IY} that $I_{\X} \subseteq I_{\Y_{\z,l}}$, so we get natural maps $J_{\z,l} \lra S/I_{\X} \lra S/I_{\Y_{\z,l}}$ giving rise to a commutative diagram
\[
\xymatrix{
\Ext^j_S(S/I_{\Y_{\z,l}}) \ar[r] & \Ext^j_S(S/I_{\X}) \ar[r] & \Ext^j_S(J_{\z,l})  \ar@{=}[d] \\
\E^j_{\z,l} \ar[r]^{\a} \ar@{^{(}->}[u] & \E^j_{\z,l}(\X) \ar[r]^{\b} \ar@{^{(}->}[u] & \Ext^j_S(J_{\z,l})
}
\] 
By the definition of $\E^j_{\z,l}(\X)$ the map $\a$ is surjective, and by the definition of $\E^j_{\z,l}$ the composition $\b\circ\a$ is an isomorphism. This shows that $\a$ is also injective, so it is an isomorphism, and it shows that $\b$ is also an isomorphism. Letting $\Y = \X \cup \{\z\}$, we get an exact sequence
\begin{equation}\label{eq:ses-JXY}
 0 \lra J_{\z,l} \lra S/I_{\X} \lra S/I_{\Y} \lra 0.
\end{equation}
Since $\b$ is surjective, we get that the maps $\Ext^j_S(S/I_{\X},S) \lra \Ext^j_S(J_{\z,l},S)$ are also surjective for $j\geq 0$. It follows that the long exact sequence for $\Ext(-,S)$ induced by (\ref{eq:ses-JXY}) splits into short exact sequences
\[0 \lra \Ext^j_S(S/I_{\Y},S) \overset{\pi^j_{\Y\to\X}}{\lra} \Ext^j_S(S/I_{\X},S) \lra \Ext^j_S(J_{\z,l},S) \lra 0.\]
Since $\b^{-1}$ gives a splitting of the surjection in the sequence above, we obtain a decomposition
\begin{equation}\label{eq:decomp-EjIX}
 \Ext^j_S(S/I_{\X},S) = \pi^j_{\Y\to\X}\left(\Ext^j_S(S/I_{\Y},S)\right) \oplus \E^j_{\z,l}(\X).
\end{equation}
Using Proposition~\ref{prop:ZYcupz} we have that $\Z(\Y) = \Z(\X) \setminus \{(\z,l)\}$, so by induction we conclude that $\E^j_{\y,u}(\Y)\simeq\Ext^{j}_S(J_{\y,u},S)$ when $(\y,u)\in\Z(\Y)$, and that
\[\Ext^j_S(S/I_{\Y},S) = \bigoplus_{(\y,u)\in\Z(\Y)} \E^j_{\y,u}(\Y).\]
Applying $\pi^j_{\Y\to\X}$ and using (\ref{eq:E^j-YtoX}), we can rewrite (\ref{eq:decomp-EjIX}) as
\[\Ext^j_S(S/I_{\X},S) = \left(\bigoplus_{(\y,u)\in\Z(\Y)} \E^j_{\y,u}(\X)\right) \oplus \E^j_{\z,l}(\X) = \bigoplus_{(\y,u)\in\Z(\X)} \E^j_{\y,u}(\X),\]
showing (\ref{eq:Ext^j=sum-E^j}) and concluding our proof.

\subsection{The proof of the Theorem on Regularity and Projective Dimension}\label{subsec:reg-projdim}
If $I_{\X}\neq S$ then we have 
\[\reg(I_{\X}) = 1 + \reg(S/I_{\X}),\mbox{ and }\pdim(I_{\X}) = -1 + \pdim(S/I_{\X}).\]
Using (\ref{eq:regM}) we obtain from (\ref{eq:ExtS/IX}) and Corollary~\ref{cor:properties-Jzl} that
\[  \reg(I_{\X}) =1 + \max_{(\ul{z},l)\in\Z(\X)}\reg(J_{\z,l}) = 1 + \max_{(\ul{z},l)\in\Z(\X)}(|\z|+l).\]
Similarly, since $\pdim(M)=\max\{j:\Ext^j_S(M,S)\neq 0\}$, we get
\[ \pdim(I_{\X}) = -1 + \max_{(\ul{z},l)\in\Z(\X)}\pdim(J_{\z,l}) = -1 + \max_{(\ul{z},l)\in\Z(\X)}(n-l),\]
concluding the proof of (\ref{eq:reg-pdim-IX}).
 
\subsection{The proof of the Theorem on Injectivity of Maps from Ext to Local Cohomology}\label{subsec:inj-loccoh}
We prove the chain of implications $(1)\Rightarrow(2)\Rightarrow(3)\Rightarrow(4)\Rightarrow(5)\Rightarrow(1)$. We assume that $\X$ consists of incomparable partitions and let $p=\min\{x_1':\x\in\X\}$. We have that $\sqrt{I}=I_p$ is the ideal defining the union of all $(p-1)$-dimensional coordinate planes and $\dim(S/I)=p-1$. In particular, if $(\z,l)\in\Z(\X)$ then $l\leq p-1$.

\noindent $(1)\Rightarrow(2)$: This is explained in a more general setting in \cite[Example~6.3]{EMS}.

\noindent $(2)\Rightarrow(3)$: Since $I$ is unmixed, we get that $I=I:I_{p-1}^{\infty}=I_{\X^{:(p-1)}}$, so by (\ref{eq:p-saturation}) it follows that no $\x\in\X$ has non-zero columns of size $\leq(p-1)$, or equivalently, every $\x\in\X$ satisfies $x_1=\cdots=x_p$.

\noindent $(3)\Rightarrow(4)$: If $(\z,l)\in\Z(\X)$ then $l\geq p-1$ by Lemma~\ref{lem:yu-in-ZY}. We have already noted that $l\leq p-1$, so $l=p-1$.

\noindent $(4)\Rightarrow(5)$: It follows from (\ref{eq:ExtS/IX}) and Corollary~\ref{cor:properties-Jzl} that $\Ext^j_S(S/I,S)=0$ for $j\neq \codim(I)$, so $S/I$ is Cohen--Macaulay.

\noindent $(5)\Rightarrow(1)$: Consider the ideal $J = I_{(W^p)}$ where $W$ is as in (\ref{def:mwW}), and note that $I\supseteq J$. Since $\sqrt{I}=\sqrt{J}=I_p$, we have $H^j_I(S)=H^j_J(S)$. By \cite[Theorem~1.1]{mustata}, the natural maps $\Ext^j_S(S/J,S) \lra H_J^j(S)$ are injective, so it is enough to show that the maps $\Ext^j_S(S/I,S)\lra \Ext^j_S(S/J,S)$ are injective. By (\ref{eq:kerExt}), this reduces to showing that $\Z(\X)\subseteq\Z(\Y)$, where $\Y=\{(W^p)\}$. Using (\ref{eq:Zx}) we get that $(\z,p-1)\in\Z(\Y)$ for every partition $\z\in\P_n$ with $z_1=\cdots=z_p\leq W-1$, so it suffices to show that any element of $\Z(\X)$ has this form. Consider any $(\z,l)\in\Z(\X)$. Since $\Ext^{n-l}_S(J_{\z,l},S)\neq 0$ and since $S/I$ is Cohen--Macaulay, it follows from (\ref{eq:ExtS/IX}) that $l=p-1$. By Remark~\ref{rem:z1--zl+1} we have that $z_1=\cdots=z_p$. If we let $c=z_1$ then we have by Definition~\ref{def:ZX}(2) that there exists $\x\in\X$ with $x'_{c+1}=p\neq 0$, so $c+1\leq x_1\leq W$, proving that $(\z,l)\in\Z(\Y)$. 

\subsection{The sequentially Cohen--Macaulay property}\label{subsec:seq-CM}

In this section we show, following the suggestion of Satoshi Murai, that for every $G$-invariant ideal $J$, the quotient $S/J$ is \defi{sequentially Cohen--Macaulay}. Letting $p=\dim(S/J)-1$, this amounts to the existence of a filtration (see \cite[Definition III.2.9]{stanley})
\[ J = J_0 \subseteq J_1\subseteq \cdots \subseteq J_p = S,\]
with the property that $J_{i+1}/J_i$ is either zero, or it is a Cohen--Macaulay module of dimension $i$. We claim that if $J=I_{\X}$ then we can take $J_i = J : I_i^{\infty} = I_{\X^{:i}}$ for $i=0,\cdots,p$. Indeed, we have by~(\ref{eq:Z-saturation}) that
\[ \Z(\X^{:0}) \supseteq \Z(\X^{:1}) \supseteq \cdots \supseteq \Z(\X^{:p}),\]
so the induced maps $\Ext^j_S(S/J_{i+1},S) \lra \Ext^j_S(S/J_{i},S)$ are injective. It follows from (\ref{eq:cokerExt}) that
\[
 \Ext^j_S(J_{i+1}/J_i,S) = \bigoplus_{(\ul{z},l)\in\Z(\X^{:i})\setminus\Z(\X^{:(i+1)})}\Ext^{j}_S(J_{\z,l},S).
\]
Using again (\ref{eq:Z-saturation}), we have that each $(\z,l)$ in the equation above satisfies $l=i$, so by Corollary~\ref{cor:properties-Jzl} we obtain $\Ext^j_S(J_{i+1}/J_i,S)=0$ for $j\neq n-i$, proving that $J_{i+1}/J_i$ is Cohen--Macaulay of dimension $i$ (or $J_i=J_{i+1}$).


\section{Ideals with a linear resolution}\label{sec:linear}

In \cite{symmetric-shifted}, the authors introduce a class of $\S_n$-invariant monomial ideals called \defi{symmetric shifted}, and they prove that these ideals have a linear minimal free resolution when they are generated in a single degree. They also leave open the problem of determining which (other) $\S_n$-invariant monomial ideals have a linear resolution. The goal of this section is to prove that no other such ideals exist, establishing the first part of the Theorem on Linear Resolutions described in the Introduction (the second part will be treated in Section~\ref{subsec:linear-pows}). We end the section explaining how our results also imply that symmetric shifted ideals that are generated in a single degree have a minimal resolution (which was proved in \cite[Section~3]{symmetric-shifted}). We recall the following \cite[Definition~1.1]{symmetric-shifted} (the minor differences in our definition below are due to a slight change in the conventions regarding partitions).

\begin{definition}\label{def:sym-shift}
 A $\S_n$-symmetric ideal $I\subseteq S$ is \defi{symmetric shifted} if for every monomial $e^{\x}\in I$ with $\x\in\P_n$, and every $1<k\leq n$ such that $x_1>x_k$, we have that $e^{\x}\cdot(e_k/e_1)\in I$.
\end{definition}

\begin{remark}\label{rem:def-sym-shift}
 We make a few observations regarding Definition~\ref{def:sym-shift} that will help streamline some of the subsequent arguments.
\begin{enumerate}
\item By \cite[Lemma~2.2]{symmetric-shifted}, it suffices to consider in Definition~\ref{def:sym-shift} only minimal generators $e^{\x}$ of $I$. 

\item If $x_1-x_k=1$ then $e^{\x}$ and $e^{\y}=e^{\x}\cdot(e_k/e_1)$ are in the same $\S_n$-orbit (since $y_1=x_k$, $y_k=x_1$, and $y_i=x_i$ for $i\neq 1,k$), so the conclusion $e^{\y}\in I$ follows from $\S_n$-invariance. Therefore, from now on we will only consider the case $x_1-x_k\geq 2$ when checking Definition~\ref{def:sym-shift}.

\item If $x_1=\cdots=x_h>x_{h+1}$ (that is, if $h$ is the height of the last column of the Young diagram of $\x$), and if $x_{g-1}>x_g=x_{g+1}=\cdots=x_k$, then $e^{\x}\cdot(e_k/e_1)$ and $e^{\y}=e^{\x}\cdot(e_g/e_h)$ are in the same $\S_n$-orbit. We may therefore assume that $g=k$ (that is, $x_{k-1}>x_k$), and the condition $e^{\x}\cdot(e_k/e_1)\in I$ is equivalent to $e^{\y}\in I$. The advantage is that $\y$ is still a partition!
\end{enumerate}
\end{remark}

\begin{remark}
In Section~\ref{subsec:linear-pows} we will use the notion of a \defi{symmetric strongly shifted} ideal to be one such that for every monomial $e^{\x}\in I$ with $\x\in\P_n$, and every $1\leq t<k\leq n$ such that $x_t>x_k$, we have that $e^{\x}\cdot(e_k/e_t)\in I$.
\end{remark}

\begin{theorem}\label{thm:sym-shift}
 If $I\subseteq S$ is an $\S_n$-invariant monomial ideal with a linear resolution then $I$ is a symmetric shifted ideal.
\end{theorem}

\begin{proof} 
 If $I$ has a linear resolution, then all its generators have the same degree, which we denote by $r$. Let $\X=\X(I)$, so that $I=I_{\X}$ and every $\x\in\X$ has $|\x|=r$. It follows from (\ref{eq:reg-pdim-IX}) that
 \begin{equation}\label{eq:ZX-sym-shift}
 |\z|+l+1 \leq r\mbox{ for every } (\z,l)\in\Z(\X).
 \end{equation}
 We assume by contradiction that $I$ is not symmetric shifted, so there exists $\x\in\P_n$ with $e^{\x}\in I$ a minimal generator (see Remark~\ref{rem:def-sym-shift}(1)), and $1<k\leq n$ such that $e^{\x}\cdot(e_k/e_1)\notin I$. We choose such an $\x$ for which $x_1$ is maximal, and define $\y$ as in Remark~\ref{rem:def-sym-shift}(3); note that in particular we are assuming that $x_{k-1}>x_k$. To get a contradiction, we will show that $e^{\y}\in I$. 
  
 By Remark~\ref{rem:def-sym-shift}(2), we may assume that $x_1-x_k\geq 2$. We define a pair $(\z,l)$ by letting $c=x_1-1$, $l=x'_{c+1}-1$ and $\z=\y(c)$, and observe that $e^{\y} = e^{\x}\cdot(e_k/e_{l+1})$ since $l+1=x'_{c+1}$ is the height of the last column of $\x$. Note also that $y'_{c+1}=l$ and therefore
 \[|\z|+l+1=|\y(c)|+y'_{c+1} +1  = |\y|+1 = |\x|+1 = r+1,\]
 so by (\ref{eq:ZX-sym-shift}) we get that $(\z,l)\notin\Z(\X)$. By construction, $\z\geq\x(c)$ and $x'_{c+1}=l+1$, so $(\z,l)$ satisfies condition (1) in Definition~\ref{def:ZX}. It follows that $(\z,l)$ must fail condition (2), that is, there exists a partition $\t\in\X$ (so that $|\t|=r$) with $\z\geq\t(c)$ and $t'_{c+1}\leq l$. We choose $\t$ to be minimal with respect to the lexicographic order.
 
 If $t_1\leq c+1$ then we claim that $\t=\y$, so that $e^{\y}\in I$ as desired. Indeed, the condition $t_1\leq c+1$ implies that $t_i\leq c+1=y_i$ for $i=1,\cdots,l$, while $t'_{c+1}\leq l$ implies $t_{l+1}\leq c$, so that 
 \[t_i = \t(c)_i\leq z_i = \y(c)_i = y_i,\mbox{ for }i=l+1,\cdots,n.\]
 This implies that $|\t|\leq |\y|=r$, but since $|\t|=r$, equality must hold everywhere, and thus $\t=\y$.
 
 We may therefore assume that $t_1>c+1=x_1$, and we let $h$ such that $t_1=\cdots=t_h>t_{h+1}$. Note that since $t'_{c+1}\leq l$, we have $h\leq l$. We claim that there exists an index $s>h$ such that either:
 \begin{itemize}
  \item $s\leq l$ and $t_1-t_s\geq 2$, or
  \item $s\geq l+1$ and $y_s>t_s$ (which implies $t_1-t_s>(c+1)-y_s\geq (c+1)-c=1$, so $t_1-t_s\geq 2$).
 \end{itemize}
If this wasn't the case, then we would get 
\[t_i = t_1 > c+1 = y_i\mbox{ for }i=1,\cdots,h,\]
\[t_i\geq t_1-1\geq c+1=y_i\mbox{ for }i=h+1,\cdots,l,\mbox{ and}\]
\[t_i\geq y_i\mbox{ for }i=l+1,\cdots,n,\]
so that $r=|\t| > |\y| = r$, a contradiction. If we take $s$ to be minimal and define $\u$ by $e^{\u} = e^{\t} \cdot (e_s/e_h)$ then $\u\in\P_n$. Moreover, since $e^{\u}$ is in the same $\S_n$-orbit as $e^{\t} \cdot (e_s/e_1)$, it follows from the maximal choice of $\x$ and the fact that $t_1>x_1$ that $e^{\u}\in I$, so $\u\in\X(I)$. Moreover, $\u$ satisfies $\z\geq\u(c)$ and $u'_{c+1}\leq l$, contradicting the minimality of $\t$ and concluding our proof.
\end{proof}

We end this section by showing that symmetric shifted ideals generated in a single degree have a linear resolution. Let $I$ be such an ideal and assume that it is generated in degree $r$. Let $\X=\X(I)$, so that $|\x|=r$ for all $\x\in\X$. Suppose by contradiction that $I$ does not have a linear resolution, so that $\reg(I)\geq r+1$. It follows that we can find $(\z,l)\in\Z(\X)$ with
\[|\z|+l+1 \geq r+1.\]
We let $c=z_1$ as usual. By Definition~\ref{def:ZX}, there exists $\x\in\X$ with $\z\geq\x(c)$ and $x'_{c+1}\leq l+1$, and we consider such an $\x$ which is lexicographically minimal. Note that by part (2) of Definition~\ref{def:ZX} we have $x'_{c+1}=l+1$. We claim that $\z>\x(c)$: suppose otherwise that $\z=\x(c)$, so that $z'_i=x'_i$ for $i=1,\cdots,c$; we have
\[ r+1\leq |\z|+l+1 = (x_1'+\cdots+x'_c) + x'_{c+1} \leq |\x| = r,\]
which is a contradiction. It follows that we can find $g>l+1$ such that $z_g>x_g$, and we choose a minimal such $g$, noting that this implies $x_{g-1}>x_g$ (since $x_{l+1}\geq c+1>z_g$ for all $g$). We let $h$, $e^{\y}$ as in Remark~\ref{rem:def-sym-shift}(3), and note that $\y\in\X(I)$ since $I$ is symmetric shifted. We also have that $\z\geq\y(c)$ and $y'_{c+1}\leq l+1$, and $\y$ is smaller than $\x$ lexicographically, contradicting the minimality of $\x$.
  
\section{Regularity of powers}\label{sec:reg-pows}

The goal of this section is to prove the Theorem on Regularity of Powers. Recall that for $\w\in\P_n$ we write 
\[\w' = (n^{a_0},h_1^{a_1},h_2^{a_2},\cdots,h_k^{a_k}),\mbox{ with }n>h_1>\cdots>h_k>0,\]
set $h_0=n$, and define
\begin{equation}\label{eq:def-bw}
b(\w) = \left(\sum_{t=1}^k (h_{t-1}-h_t)\cdot(a_t-1)\right) + (h_k-1)\cdot (a_k-1).
\end{equation}
We will prove the following.
\begin{theorem}\label{thm:reg-powers}
With the notation above, we have that
\begin{equation}\label{eq:reg-Iwd}
\reg(I_{\w}^d) = d\cdot |\w| + b(\w)\mbox{ for }d\gg 0.
\end{equation}
Moreover, $I_{\w}^d$ has a linear resolution for $d\gg 0$ if and only if $w_i-w_{i+1}\leq 1$ for all $i=1,\cdots,n-1$, in which case $I_{\w}^d$ is symmetric strongly shifted.
\end{theorem}

\begin{example}\label{ex:reg-pows}
 Take $n=4$ and $\w=(2,1,0,0)$. We have that $k=2$, $h_1=2$, $h_2=1$, $a_0=0$, $a_1=a_2=1$, and in particular $b(\w)=0$. The Betti tables for $I_{\w},I_{\w}^2$ and $I_{\w}^3$ are given respectively by:
 \[
 \begin{array}{c|cccc}
      &0&1&2&3\\ \hline
      \text{3}&12&18&4&-\\ 
      \text{4}&-&-&4&-\\
      \text{5}&-&-&-&1\\
\end{array},
\qquad
 \begin{array}{c|cccc}
      &0&1&2&3\\ \hline
      \text{6}&64&152&117&24\\ 
      \text{7}&-&-&-&4\\
\end{array},
\quad\mbox{ and }\quad
 \begin{array}{c|cccc}
      &0&1&2&3\\ \hline
      \text{9}&180&474&420&125\\ 
\end{array}.
\]
One can check that in this case we have $I_{\w}^3 = I_{\X}$, where
\[ 
\begin{aligned}
\X = \{&(6,3),(5,4),(6,2,1),(5,3,1),(4^2,1),(5,2^2),(4,3,2),(3^3),\\
&(6,1^3),(5,2,1^2),(4,3,1^2),(4,2^2,1),(3^2,2,1),(3,2^3)\}.
\end{aligned}
\]
A more compact description of $\X$ is as the set of all partitions of size $9$ in $\P_4$ that are smaller in the dominance order than $(6,3)$. As explained in \cite[Remark~1.3]{symmetric-shifted}, this implies that $I_{\w}^3$ is symmetric strongly shifted. By contrast, we have that $I_{\w}^2 = I_{\Y}$ where
\[\Y = \{(4,2),(3^2),(3,2,1),(2^3),(3,1^3),(2^2,1^2)\}\]
which contains all the partitions of size $6$ in $\P_4$ dominated by $(4,2)$, except for $(4,1^2)$!
\end{example}

We let
\[ \X_{\w}^d = \left\{ \x\in\P_n : \mbox{ there exists }\s_j \in \S_n, j=1,\cdots,d,\mbox{ such that }\x=\sum_{j=1}^d \s_j(\w) \right\},\]
and note that $I_{\w}^d = I_{\X_{\w}^d}$ for all $d\geq 1$. In light of (\ref{eq:reg-pdim-IX}), we have to check that
\[\max \left\{ |\z| + l + 1 : (\z,l) \in \Z(\X_{\w}^d)\right\} = d\cdot |\w| + b(\w)\mbox{ for }d\gg 0.\]
Our strategy will be to translate the containment $(\z,l) \in \Z(\X_{\w}^d)$ into feasibility conditions on a high-multiplicity partitioning problem in Section~\ref{subsec:partitioning}. We then use the results of \cite{raicu-partitioning} to characterize such feasibility conditions and obtain a quick proof of (\ref{eq:reg-Iwd}) in Section~\ref{subsec:proof-reg-pows}. We prove the last assertion of Theorem~\ref{thm:reg-powers} in Section~\ref{subsec:linear-pows}.

\subsection{The relationship with partitioning problems}\label{subsec:partitioning}

Following \cite{raicu-partitioning}, we think of $\w$ as a tuple of \defi{ball-weights}, and for $d>0$ we consider a collection of $d\cdot n$ balls, with $d$ of weight $w_i$ for each $i=1,\cdots,n$. We encode them as elements of a multi-set
\[\mc{B} = \{w_1,\cdots,w_1,\cdots,w_i,\cdots,w_i,\cdots,w_n,\cdots,w_n\},\]
where each $w_i$ is repeated $d$ times. For a tuple $\ul{C}=(C_1,\cdots,C_n)$ of \defi{capacities} we consider the problem $\mf{BP}(d,\ul{C};\w)$ of partitioning $\mc{B}$ as
\begin{equation}\label{eq:partition-B}
\mc{B} = \mc{B}_1 \sqcup \cdots \sqcup \mc{B}_i \sqcup \cdots \sqcup \mc{B}_n,
\end{equation}
where each $\mc{B}_i$ has exactly $d$ elements, and
\begin{equation}\label{eq:constraints-partitioning}
w(\mc{B}_i):=\sum_{w\in\mc{B}_i} w \leq C_i,\mbox{ for }i=1,\cdots,n.
\end{equation}
Thinking of each $\mc{B}_i$ as a bin with capacity $C_i$, this is the same as assigning the balls to bins in such a way that each $\mc{B}_i$ is assigned $d$ balls without exceeding its capacity. We always assume that $\ul{C}$ (just as $\w$) is non-increasing ($C_1\geq\cdots\geq C_n$). A partition $\mc{B}_{\bullet}$ is said to be \defi{$r$-feasible} if 
\[w(\mc{B}_i)\leq C_i\mbox{ for }i=r+1,\cdots,n,\]
and it is \defi{feasible} if it satisfies (\ref{eq:constraints-partitioning}), that is, if it is $0$-feasible. If there exists an $r$-feasible partition $\mc{B}_{\bullet}$ then we say that the problem $\mf{BP}(d,\ul{C};\w)$ itself is \defi{$r$-feasible} (or \defi{feasible} when $r=0$). The goal of this section is to characterize the containment $(\z,l)\in\Z(\X_{\w}^d)$ in terms of feasibility conditions on $\mf{BP}(d,\z;\w)$. More precisely, we prove the following.

\begin{proposition}\label{prop:ZX-vs-feasibility}
 Suppose that $\z\in\P_n$ and $0\leq l<n$ is such that $z_1=\cdots=z_{l+1}$. We have that
 \[(\z,l)\in\Z(\X_{\w}^d) \Longleftrightarrow \mf{BP}(d,\z;\w)\mbox{ is $(l+1)$-feasible but not $l$-feasible.}\]
\end{proposition}

\begin{proof} We first note that the condition that $\x$ can be represented as
\[\x=\sum_{j=1}^d \s_j(\w) \mbox{ for permutations }\s_j\in\S_n,\ j=1,\cdots,d,\]
is equivalent to the condition that there exists a partition $\mc{B}_{\bullet}$ with each $\mc{B}_i$ containing exactly $d$ elements, and $w(\mc{B}_i)=x_i$ for $i=1,\cdots,n$. To see this, note that given permutations $\s_j\in\S_n$, we can form $\mc{B}_{\bullet}$ by placing for each $j=1,\cdots,d$ and $i=1,\cdots,n$, a ball of weight $w_i$ into the bin $\mc{B}_{\s_j(i)}$. Conversely, it follows from Hall's Marriage Theorem \cite{hall} that given $\mc{B}_{\bullet}$, we can find a permutation $\s_d\in\S_n$ with the property that $\mc{B}_{\s_d(i)}$ contains a ball of weight $w_i$ for all $i=1,\cdots,n$. Removing one such ball for each $i$ and applying induction on $d$ allows us to construct $\s_{d-1},\cdots,\s_1\in\S_n$ with the desired property.

We next let $c=z_1$ and observe that the conditions that $\z\geq\x(c)$ and $x'_{c+1}\leq l+1$ are equivalent to the inequalities $x_i\leq z_i$ for all $i>l+1$. Furthermore, we have $x_{l+1}\leq c=z_{l+1}$ precisely when $x'_{c+1}\leq l$. It follows from the discussion in the previous paragraph that condition (1) in Definition~\ref{def:ZX} is equivalent to the fact that $\mf{BP}(d,\z;\w)$ is $(l+1)$-feasible, while condition (2) is satisfied if and only if $\mf{BP}(d,\z;\w)$ is not $l$-feasible.
\end{proof}

We record one more fact to be used in the proof of the Theorem on Regularity of Powers. Given a tuple $\u\in\bb{Z}^n$ and an integer $1\leq r\leq n$ we write $\u^{\geq r}$ for the truncation $(u_r,u_{r+1},\cdots,u_n)$.

\begin{lemma}\label{lem:zl+1-large-l-feasible}
 If $0\leq l<n$ and $\mf{BP}(d,\z;\w)$ is $(l+1)$-feasible but not $l$-feasible, then $z_{l+1}< d\cdot w_{l+1}$.
\end{lemma}

\begin{proof}
 Using \cite[Lemma~7.1]{raicu-partitioning} with $j=0$ we obtain that $\mf{BP}(d,\z^{\geq l+1};\w^{\geq l+1})$ is $1$-feasible but not feasible. Let $\mc{B}_1\sqcup\cdots\sqcup\mc{B}_{n-l}$ denote a $1$-feasible partition of the multi-set $\mc{B}=\{w_{l+1},\cdots,w_{l+1},\cdots,w_n,\cdots,w_n\}$, where each $w_i$ is repeated $d$ times for $i=l+1,\cdots,n$. We have that $w(\mc{B}_i) \leq z_{l+i}$ for $i=2,\cdots,n-l$, and since $\mc{B}_{\bullet}$ is not feasible by hypothesis, we have that $w(\mc{B}_{l+1}) > z_{l+1}$. Since every ball in the multi-set $\mc{B}$ has weight at most $w_{l+1}$, and since $\mc{B}_{l+1}$ contains exactly $d$ balls, it follows that $w(\mc{B}_{l+1})\leq d\cdot w_{l+1}$, from which the desired conclusion follows.
\end{proof}

\subsection{The linear function computing regularity of powers}\label{subsec:proof-reg-pows}
In \cite[Theorem~1.11]{raicu-partitioning} we show that for $d\geq n$ we can find $\z\in\P_n$ with $|\z|=d\cdot |\w| + b(\w)-1$ such that $\mf{BP}(d,\z;\w)$ is $1$-feasible but not $0$-feasible. Using Proposition~\ref{prop:ZX-vs-feasibility} it follows that $(\z,0)\in\Z(\X_{\w}^d)$, so by (\ref{eq:reg-pdim-IX}) we conclude that
\[\reg(I_{\w}^d) \geq |\z|+1 = d\cdot |\w| + b(\w)\mbox{ for }d\geq n.\]

Suppose by contradiction that $\reg(I_{\w}^d) > d\cdot |\w| + b(\w)$ for some $d\gg 0$. It follows from (\ref{eq:reg-pdim-IX}) that there exists $(\z,l)\in\Z(\X_{\w}^d)$ such that
\begin{equation}\label{eq:ineq-zl-dwb}
 |\z|+l\geq d\cdot|\w|+b(\w).
\end{equation}
By Proposition~\ref{prop:ZX-vs-feasibility} we have that $\mf{BP}(d,\z;\w)$ is $(l+1)$-feasible but not $l$-feasible. If $l=0$ then we know by \cite[Theorem~1.10]{raicu-partitioning} that if $|\z|\geq d\cdot|\w|+b(\w)$ and $\mf{BP}(d,\z;\w)$ is $1$-feasible, then $\mf{BP}(d,\z;\w)$ is also $0$-feasible. This is a contradiction, allowing us to assume that $l>0$. Applying \cite[Lemma~7.1]{raicu-partitioning} with $j=0$ we have that $\mf{BP}(d,\z^{\geq l+1};\w^{\geq l+1})$ is $1$-feasible but not $0$-feasible, so by \cite[Theorem~1.10]{raicu-partitioning} applied to $\mf{BP}(d,\z^{\geq l+1};\w^{\geq l+1})$ we conclude that (using the definition (\ref{eq:def-bw}) for the truncation $\w^{\geq l+1}$)
\[ z_{l+1}+\cdots+z_n < d\cdot(w_{l+1}+\cdots+w_n) + b(\w^{\geq l+1}).\]
We write $c=z_1$ and recall that by Remark~\ref{rem:z1--zl+1} we have $z_1=\cdots=z_{l+1}=c$. Adding $z_1+\cdots+z_l + l=l\cdot c+l$ to both sides of the inequality above and using (\ref{eq:ineq-zl-dwb}) we obtain
\[d\cdot|\w|+b(\w) \leq |\z|+l < d\cdot(w_{l+1}+\cdots+w_n) + b(\w^{\geq l+1}) + l\cdot c + l,\]
which after simplifications yields
\begin{equation}\label{eq:ineq-l-c-delta}
d\cdot(w_1+\cdots+w_l) < l\cdot c + \Delta,
\end{equation}
where $\Delta=b(\w^{\geq l+1})+l-b(\w)$ is some constant, depending only on $\w$ and $l$, but not on $d$. It follows from Lemma~\ref{lem:zl+1-large-l-feasible} that $c<d\cdot w_{l+1}$, so (\ref{eq:ineq-l-c-delta}) implies that
\[ d\cdot(w_1+\cdots+w_l - l\cdot w_{l+1}) < \Delta.\]
Since $w_1,\cdots,w_l\geq w_{l+1}$ and $d\gg 0$, this is only possible when $w_1=\cdots=w_{l+1}$, which forces $l<h_k$. This implies that $b(\w)-b(\w^{\geq l+1})=l\cdot(a_k-1)\geq 0$, and therefore $\Delta\leq l$. Combining this with the inequality $c\leq d\cdot w_{l+1}-1$, it follows from (\ref{eq:ineq-l-c-delta}) that
\[ d\cdot(w_1+\cdots+w_l) < l\cdot(d\cdot w_{l+1}-1) + \Delta = d\cdot l\cdot w_{l+1} + (\Delta-l) \leq d\cdot l\cdot w_{l+1},\]
contradicting the fact that $w_1=\cdots=w_l=w_{l+1}$ and concluding our proof of (\ref{eq:reg-Iwd}).

\subsection{Powers with a linear resolution}\label{subsec:linear-pows}
Since $I_{\w}^d$ is generated in degree $d\cdot|\w|$, it follows that it has a linear free resolution if and only if $\reg(I_{\w}^d)=d\cdot|\w|$. For $d\gg 0$, this is equivalent by (\ref{eq:reg-Iwd}) with the fact that $b(\w)=0$. Since $h_{t-1}-h_t>0$ for every $t=1,\cdots,k$, this is further equivalent to the fact that $a_1=\cdots=a_k=1$, which is finally equivalent to the requirement that $w_i-w_{i+1}\leq 1$ for all $i=1,\cdots,n-1$.

To finish the proof of Theorem~\ref{thm:reg-powers} we assume that $w_i-w_{i+1}\leq 1$ for all $i=1,\cdots,n-1$ and show that for $d\gg 0$, the set $\X_{\w}^d$ consists of all the partitions of size $d\cdot|\w|$ that are dominated by $d\cdot\w$. By \cite[Remark~1.3]{symmetric-shifted}, this is enough to conclude that $I_{\w}^d$ is symmetric strongly shifted. Choose any partition $\ul{C}\in\P_n$ with $|\ul{C}| = d\cdot |\w|$, which is dominated by $d\cdot\w$. This means that
\begin{equation}\label{eq:dominance}
 C_i + \cdots + C_n \geq d\cdot(w_i+\cdots+w_n)\mbox{ for all }i=1,\cdots,n.
\end{equation}
Notice that our assumption on $\w$ guarantees that $b(\w^{\geq i}) = 0$ for all $i=1,\cdots,n$, so by \cite[Theorem~1.7]{raicu-partitioning} we conclude that the partitioning problem $\mf{BP}(d,\ul{C};\w)$ is feasible. Let $\mc{B}_{\bullet}$ be a solution, and note that since
\[ d\cdot |\w| = C_1+\cdots+C_n\leq w(\mc{B}_1) + \cdots + w(\mc{B}_n) = w(\mc{B}) = d\cdot|\w|,\]
we must have equality throughout, which is possible only when $w(\mc{B}_i)=C_i$ for all $i=1,\cdots,n$. Using the dictionary between elements of $\X_{\w}^d$ and bin-weights established in the proof of Proposition~\ref{prop:ZX-vs-feasibility}, we conclude that $\ul{C}\in\X_{\w}^d$. Conversely, we have seen that any $\ul{C}\in\X_{\w}^d$ leads to a feasible problem $\mf{BP}(d,\ul{C};\w)$. Since the bins $\mc{B}_i,\cdots,\mc{B}_n$ must contain collectively a total of $d\cdot (n-i+1)$ balls, whose total weight can be no smaller than the sum of the smallest $d\cdot(n-i+1)$ elements of the multi-set $\mc{B}$, namely $d\cdot(w_i + \cdots + w_n)$, it follows that (\ref{eq:dominance}) must hold, that is, $\ul{C}$ is dominated by $d\cdot\w$.

\section{Varying the number of variables}\label{sec:vary-n}

Recall that for a subset $\X\subset\P$ we write
\[ \X_n = \{\x\in\X : \x\mbox{ has at most }n\mbox{ parts}\} \subseteq \P_n.\]
The goal of this section is to study $\reg(I_{\X_n})$ and $\pdim(I_{\X_n})$ as functions of $n$ when $n\gg 0$, recovering recent results of Murai \cite{murai}. More precisely, we show the following (the Theorem on Invariant Chains of Ideals).

\begin{theorem}\label{thm:reg-linear-n}
 Let $\X$ denote a finite non-empty set of pairwise incomparable partitions, and define
 \[ m = \max\{ i : x_i\neq 0\mbox{ for some }\x\in\X\},\quad w = \min\{x_1 : \ul{x}\in\mc{X}\},\quad\mbox{ and }\quad W = \max\{x_1 : \ul{x}\in\mc{X}\}.\]
 If we let $\Y = \{ \x-\x(w-1) : \x\in\X\}$, then we have the following.
\begin{enumerate}
 \item There exists a constant $C$ such that $\reg(I_{\Y_n})=C$ for $n\geq m$.
 \item We have $\reg(I_{\X_n}) = (w-1)\cdot n + C$ for $n\geq \max\bigl( m,(m-1)\cdot(W-w+2)-C\bigr)$.
\end{enumerate}
\end{theorem}

\begin{example}\label{ex:regIn-sharp}
 Both conclusions in Theorem~\ref{thm:reg-linear-n} are sharp. 
 \begin{itemize}
 \item For (1), consider $\X=\{(1^m)\}$, so that $I_{\X_n}=S$ for $n<m$ has regularity $0$, and $I_{\X_n}$ is the ideal generated by all square-free monomials of degree $m$ when $n\geq m$, whose regularity is $C=m$.
 \item For (2), consider $\X=\{(2,1^{m-1}),(W^{m-1})\}$, so that $I_{\Y_n}$ is the maximal ideal of $S$, whose regularity is $C=1$. Note also that $w=2$. It can be checked that
 \[ \reg(I_{\X_n}) = W\cdot(m-1) > n + 1\mbox{ for }m\leq n < (m-1)\cdot W-1,\]
 and that $(\z,l)=\bigl( (W-1)^{m-1},m-2\bigr)\in\Z(\X_n)$ provides the maximal value for $|\z|+l+1$ in (\ref{eq:reg-pdim-IX}). One can also check that $\reg(I_{\X_n})=n+1$ for $n\geq(m-1)\cdot W-1$, with $(\z,l) = \bigl((1^n),0\bigr)\in\Z(\X_n)$ maximizing $|\z|+l+1$ in (\ref{eq:reg-pdim-IX}). For a specific example, take $m=W=3$: using Macaulay2, one has that for $n=4$, the Betti table of $I_{\Y_n}$ is
 \[
 \begin{array}{c|cccc}
      &0&1&2&3\\ \hline
      \text{4}&12&20&10&-\\ 
      \text{5}&-&-&-&1\\
      \text{6}&6&12&6&-\\
\end{array}
\]
so the regularity is $6>n+1$. If we take instead $n=5$ then the Betti table is
\[
 \begin{array}{c|ccccc}
      &0&1&2&3&4\\ \hline
      \text{4}&30&70&55&10&-\\ 
      \text{5}&-&-&-&5&-\\
      \text{6}&10&30&30&10&1\\
\end{array}
\]
so the regularity is $6=n+1$.
 \end{itemize}
\end{example}

We begin by noting that $\x-\x(w-1)$ is the partition obtained from $\x$ by removing the first $(w-1)$ columns in its Young diagram. More precisely, we have for every partition $\u\in\P$ and $r\geq 0$ that
\begin{equation}\label{eq:u-ur}
\v=\u-\u(r) \Longleftrightarrow v'_i = u'_{i+r} \mbox{ for all }i\geq 1.
\end{equation}

\begin{lemma}\label{lem:square-free-Y}
 Suppose that $\Y\subset\P_n$ is a set of partitions containing $(1^p)$ for some $p\leq n$. For every $(\z,l)\in\Z(\Y)$ we have that $z_p=0$, that is, $\ul{z}$ has at most $(p-1)$ parts.
\end{lemma}

\begin{proof}
 Suppose by contradiction that there exists $(\z,l)\in\Z(\Y)$ such that $z_p\neq 0$, and let $c=z_1$. Since $\z$ is not the empty partition, it follows that $c\geq 1$, so if we let $\y=(1^p)$ then $\y=\y(c)$ and $y'_{c+1}=0$. Moreover, since $z_p\neq 0$ we have that $\z\geq\y=\y(c)$, and since $y'_{c+1}<l+1$, condition (2) in Definition~\ref{def:ZX} is violated. This gives the desired contradiction, concluding the proof.
\end{proof}

\begin{corollary}\label{cor:square-free-Y}
 If $\Y\subset\P$ is a set of partitions containing $(1^p)$ then $\Z(\Y_n)$ is independent on $n$ for $n\geq p$. In particular, $\reg(I_{\Y_n})$ is constant for $n\geq p$.
\end{corollary}

\begin{proof}
 We may assume without loss of generality that $\Y$ consists of pairwise incomparable partitions, and that $p$ is minimal with the property that $(1^p)\in\Y$. Any partition $\y$ with more than $p$ parts has the property that $\y\geq(1^p)$, so $\y\not\in\Y$. This implies that $\Y_p=\Y_n$, and therefore $\Z(\Y_p)\subseteq\Z(\Y_n)$ for all $n\geq p$. To prove equality we need to show that for every $(\z,l)\in\Z(\Y_n)$ we have $\z\in\P_p$, which follows from Lemma~\ref{lem:square-free-Y}. The fact that $\reg(I_{\Y_n})$ is constant for $n\geq p$ follows now from (\ref{eq:reg-pdim-IX}).
\end{proof}

\begin{lemma}\label{lem:zx-vs-uy}
 Suppose that $n\geq m$ and let $\z\in\P_n$ with $z_m\geq w-1$. Let $\u=\z-\z(w-1)$, $c=z_1$ and $d=u_1=c-(w-1)$. Consider $\x\in\X_n$ and let $\y=\x-\x(w-1)\in\Y_n$. We have that $x'_{c+1}=y'_{d+1}$ and
 \[\z\geq\x(c) \Longleftrightarrow \u\geq\y(d).\]
\end{lemma}

\begin{proof}
 Note that by (\ref{eq:u-ur}) we have $y'_{d+1}=x'_{d+1+w-1}=x'_{c+1}$, proving the first assertion. Moreover, we have
 \begin{equation}\label{eq:zgeqxc-equiv}
 \z\geq\x(c) \Longleftrightarrow z'_i\geq x'_i\mbox{ for }i\leq c.
 \end{equation}
 Since $z_m\geq w-1$ we get that $z'_i\geq m$ for $i\leq w-1$. Since $\x$ has at most $m$ parts, $x'_i\leq m$ for all $i$, and in particular $x'_i\leq z'_i$ for $i\leq w-1$. We obtain
 \[\z\geq\x(c) \overset{(\ref{eq:zgeqxc-equiv})}{\Longleftrightarrow}z'_i\geq x'_i\mbox{ for }w\leq i\leq c \overset{(\ref{eq:u-ur})}{\Longleftrightarrow}u'_{i-w+1}\geq y'_{i-w+1}\mbox{ for }w\leq i\leq c \Longleftrightarrow u'_j\geq y'_j\mbox{ for }j\leq d,\]
 which in turn is equivalent to $\u\geq\y(d)$ by the analogue of (\ref{eq:zgeqxc-equiv}), concluding the proof.
\end{proof}

\begin{proof}[Proof of Theorem~\ref{thm:reg-linear-n}]
 To verify conclusion (1), we consider $\x\in\X$ with $x_1=w$ and let $\y=\x-\x(w-1)\in\Y$. Since $y_1=1$ it follows that $\y=(1^p)$ for some $p$, and since $\x$ has at most $m$ parts, we get $p\leq m$. It follows from Corollary~\ref{cor:square-free-Y} that $\reg(I_{\Y_n})$ is constant for $n\geq p$, and in particular for $n\geq m$, as desired.

To prove conclusion (2), we verify that
\begin{equation}\label{eq:squeeze-reg-IX}
 n\cdot(w-1)+C \leq \reg(I_{\X_n}) \leq \max\bigl(n\cdot(w-1)+C,n\cdot(w-2)+(m-1)\cdot(W-w+2)\bigr)\mbox{ for }n\geq m.
\end{equation}
When $n\geq (m-1)\cdot(W-w+2)-C$ we get that $n\cdot(w-1)+C\geq n\cdot(w-2)+(m-1)\cdot(W-w+2)$, which together with (\ref{eq:squeeze-reg-IX}) implies $\reg(I_{\X_n})=n\cdot(w-1)+C$, as desired. 

To prove the first inequality in (\ref{eq:squeeze-reg-IX}) we let $(\u,l)\in\Z(\Y_n)$ such that $C=\reg(I_{\Y_n}) = |\u|+l+1$ and define
\[\z=\bigl((w-1)^n\bigr) + \u.\]
We let $d=u_1$, $c=z_1=d+(w-1)$. Since $n\geq m$, we have $z_m\geq w-1$, and since every $\y\in\Y_n$ has the form $\y=\x-\x(w-1)$ for some $\x\in\X_n$, it follows from Lemma~\ref{lem:zx-vs-uy} that
\begin{equation}\label{eq:equiv-ZX-ZY}
 (\z,l)\in\Z(\X_n) \Longleftrightarrow (\u,l)\in\Z(\Y_n),
\end{equation}
and in particular
\[ \reg(I_{\X_n}) \geq |\z|+l+1 = n\cdot(w-1)+|\u|+l+1 = n\cdot(w-1) + C.\]
For the second inequality in (\ref{eq:squeeze-reg-IX}) we choose $(\z,l)\in\Z(\X_n)$ with $\reg(I_{\X_n}) = |\z|+l+1$, and let $\u=\z-\z(w-1)$. If $z_m\geq w-1$ then it follows as before from Lemma~\ref{lem:zx-vs-uy} that (\ref{eq:equiv-ZX-ZY}) holds, so
\[ \reg(I_{\X_n}) = |\z|+l+1 \leq n\cdot(w-1)+|\u|+l+1 \leq n\cdot(w-1) + C.\]
If $z_m\leq w-2$ then $z'_i\leq m-1$ for $i\geq w-1$. Since $x'_{c+1}=l+1$ we get that $x_1\geq c+1$ so $c\leq W-1$, and moreover we have $l+1\leq m-1$. This yields
\[ \reg(I_{\X_n}) = |\z| + l + 1 = \left(\sum_{i=1}^{w-2} z_i'\right) + \left(\sum_{i=w-1}^{W-1} z_i'\right) + (l+1) \leq n\cdot(w-2) + (m-1)\cdot(W-w+1) + (m-1),\]
proving the second inequality in (\ref{eq:squeeze-reg-IX}) and concluding the proof.
\end{proof}

\section*{Acknowledgments} 
The author would like to thank Eric Ramos for interesting discussions that started this project, and Satoshi Murai for suggesting some of the questions studied here, and for very useful comments and corrections on earlier versions of the manuscript. Experiments with the computer algebra software Macaulay2 \cite{M2} have provided numerous valuable insights. The author acknowledges the support of the Alfred P. Sloan Foundation, and of the National Science Foundation Grant No.~1901886.


\end{document}